\documentclass[11pt,reqno]{amsproc}
\title[]{Global Regularity for Nernst-Planck-Navier-Stokes Systems with Mixed Boundary Conditions}

\author{Fizay-Noah Lee}
\address{Program in Applied and Computational Mathematics, Princeton University, Princeton, NJ 08544}
\email{fizaynoah@princeton.edu}

\usepackage{setspace}
\usepackage[margin=1in]{geometry}
\usepackage{amsmath, amsthm, amssymb}
\usepackage{times}
\usepackage{color}
\usepackage{hyperref}
\usepackage{comment}

\newcommand{\pa}{\partial}

\newcommand{\la}{\label}
\newcommand{\fr}{\frac}
\newcommand{\na}{\nabla}
\newcommand{\be}{\begin{equation}}
\newcommand{\ee}{\end{equation}}
\newcommand{\bal}{\begin{aligned}}
\newcommand{\eal}{\end{aligned}}
\newcommand{\ba}{\begin{array}{l}}
\newcommand{\ea}{\end{array}}

\newtheorem{thm}{Theorem}
\newtheorem{prop}{Proposition}
\newtheorem{lemma}{Lemma}
\newtheorem{rem}{Remark}
\newtheorem{defi}{Definition}

\renewcommand{\div}{{\mbox{div}\,}}
\newcommand{\D}{\Delta}

\newcommand{\tci}{{\tilde{c}_i}}
\newcommand{\tcj}{{\tilde{c}_j}}

\date{today}

\begin{document}
\begin{abstract}
We consider electrodiffusion of ions in fluids, described by the Nernst-Planck-Navier-Stokes system, in three dimensional bounded domains, with mixed blocking (no-flux) and selective (Dirichlet) boundary conditions for the ionic concentrations and Robin boundary conditions for the electric potential, representing the presence of an electrical double layer. We prove global existence of strong solutions for large initial data in the case of two oppositely charged ionic species. The result hold unconditionally in the case where fluid flow is described by the Stokes equations. In the case of Navier-Stokes coupling, the result holds conditionally on Navier-Stokes regularity. We use a simplified argument to also establish global regularity for the case of purely blocking boundary conditions for the ionic concentrations for two oppositely charged ionic species and also for more than two species if the diffusivities are equal and the magnitudes of the valences are also equal.
\end{abstract}

\noindent
\maketitle
\section{Introduction}
We study the \textit{Nernst-Planck-Navier-Stokes} (NPNS) system in a connected, bounded domain $\Omega\subset\mathbb{R}^3$ with smooth boundary, which models the electrodiffusion of ions in a fluid in the presence of boundaries. The ions diffuse under the influence of their own concentration gradients and are transported by the fluid and an electric field, which is generated by the local charge density and an externally applied potential. The fluid is forced by the electrical force exerted by the ionic charges. The time evolution of the ionic concentrations is determined by the \textit{Nernst-Planck} equations,
\be
\pa_t c_i+u\cdot\na c_i=D_i\div(\na c_i+z_ic_i\na\Phi),\quad i=1,...,m\la{np}
\ee
coupled to the Poisson equation
\be
-\epsilon\D\Phi=\sum_{i=1}^m z_ic_i=\rho\la{pois}
\ee
and to the \textit{Navier-Stokes} system,
\be
\pa_t u+u\cdot\na u -\nu\D u+\na p=-K\rho\na\Phi,\quad \div u=0\la{nse}
\ee
or to the \textit{Stokes} system
\be
\pa_t u-\nu\D u+\na p=-K\rho\na\Phi,\quad \div u=0.\la{stokes}
\ee
In this latter case we refer to the system as the \textit{Nernst-Planck-Stokes} (NPS) system.

The function $c_i$ is the local ionic concentration of the $i$-th species, $u$ is the fluid velocity, $p$ is the pressure, $\Phi$ is a rescaled electrical potential, and $\rho$ is the local charge density. The constant $z_i\in\mathbb{R}$ is the ionic valence of the $i$-th species. The constants $D_i>0$ are the ionic diffusivities, and $\epsilon>0$ is a rescaled dielectric permittivity of the solvent, and it is proportional to the square of the Debye length $\lambda_D$, which is the characteristic length scale of the electrical double layer in a solvent \cite{rubibook}. The constant $K>0$ is a coupling constant given by the product of Boltzmann's constant $k_B$ and the temperature $T_K$. Finally, $\nu>0$ is the kinematic viscosity of the fluid. The dimensional counterparts of $\Phi$ and $\rho$ are given by $(k_BT_K/e)\Phi$ and $e\rho$, respectively, where $e$ is elementary charge.

For the ionic concentrations $c_i$ we consider \textit{blocking} (no-flux) boundary conditions,
\be
(\pa_n c_i(x,t)+z_ic_i(x,t)\pa_n\Phi(x,t))_{|\pa\Omega}=0 \la{bl}
\ee
where $\pa_n$ is the outward normal derivative along $\pa\Omega$. This boundary condition represents a surface that is impermeable to the $i$-th ionic species. For regular enough solutions, blocking boundary conditions imply that the total concentration $\int_\Omega c_i\,dx$ is conserved as can formally be seen by integrating (\ref{np}) over $\Omega$. For $c_i$, we also consider \textit{selective} (Dirichlet) boundary conditions,
\be
c_i(x,t)_{|\pa\Omega}=\gamma_i>0,\la{DI}
\ee
which, in electrochemistry \cite{davidson,rubishtil}, represents an ion-selective (permselective) membrane that maintains a fixed concentration of ions.

The boundary conditions for the Navier-Stokes (or Stokes) equations are \textit{no-slip},
\be
u(x,t)_{|\pa\Omega}=0.\la{noslip}
\ee
The boundary conditions for $\Phi$ are inhomogeneous Robin,
\be
(\pa_n\Phi(x,t)+\tau\Phi(x,t))_{|\pa\Omega}=\xi(x).\la{robin}
\ee
This boundary condition represents the presence of an electrical double layer at the interface of a solvent and a surface \cite{prob,rubibook}. The Robin boundary conditions are derived based on the fact that the double layer has the effect of a plate capacitor. The constant $\tau>0$ represents the capacitance of the double layer, and $\xi:\partial\Omega\to\mathbb{R}$ is a smooth function that represents an externally applied potential on the boundary (see also \cite{bothe,fischer,gajewski,lee} where the same boundary conditions are used in similar contexts).

In this paper, we discuss the question of global regularity of solutions of NPNS and NPS. The NPNS system is a semilinear parabolic system, and in general, such systems can blow up in finite time. For example, The Keller-Segel equations, which share some common features with NPNS (e.g. the dissipative structure, Section {\ref{de}}), are known to admit solutions that blow up in finite time, even in two dimensions, for large initial conditions \cite{bedro}. The NPNS system includes the Navier-Stokes equations, where the question of large data global regularity, as is well known, is unresolved in three dimensions \cite{cf}. So for NPNS, we cannot expect at this stage to obtain affirmative results on unconditional global regularity. However, global regularity in three dimensions for the NPS system or even the Nernst-Planck equations, not coupled to fluid flow, is still an open problem except in some special cases, the main obstacle being control of the nonlinear term $\div(c_i\na\Phi)$.

In both the physical and mathematical literature, many different boundary conditions are considered for the concentrations $c_i$ and the electric potential $\Phi$, all with different physical meanings. The choice of boundary conditions makes a large difference not only when it comes to determining global regularity, but also in characterizing long time behavior. Global existence and stability of solutions to the uncoupled Nernst-Planck equations is obtained in \cite{biler,biler2,choi,gajewski} for blocking boundary conditions in two dimensions. The full NPNS system is discussed in \cite{schmuck} where the electric potential is treated as a superposition of an internal potential, determined by the charge density $\rho$ and homogeneous Neumann boundary conditions, and an external, prescribed potential. In this case global weak solutions are obtained in both two and three dimensions. The case of blocking and selective (Dirichlet) boundary conditions for the concentrations and Dirichlet boundary conditions for the potential are considered in \cite{ci} for two dimensions, and the authors obtain global existence of strong solutions and, in the case of blocking or \emph{uniformly} selective boundary conditions, unconditional stability. In \cite{np3d}, these results are extended to three dimensions, for initial conditions that are small perturbations of the steady states. In \cite{bothe,fischer} the authors consider blocking boundary conditions for $c_i$ and Robin boundary conditions for $\Phi$, as we do in this paper, and obtain global regularity and stability in two dimensions and global weak solutions in three dimensions. In \cite{cil} global regularity in three dimensions is obtained in the case of Dirichlet boundary conditions for both the concentrations and the potential. In \cite{liu}, the authors establish global existence of weak solutions in the case of no boundaries, $\mathbb{R}^3$.

As established and used effectively in the works referred to in the previous paragraph, the NPNS system, equipped with blocking or uniformly selective (c.f. \cite{ci}) boundary conditions for $c_i$, comes with a dissipative structure (Section \ref{de}), which in particular leads to stable asymptotic behaviors. Deviations from blocking or uniformly selective boundary conditions are known to lead, in general, to instabilities when a large enough electric potential drop is imposed across the spatial domain (e.g. narrow channel). These so-called electrokinetic instabilities (EKI) are observed both experimentally and numerically, and verified analytically through simplified models \cite{davidson,pham,rubinstein,rubizaltz,zaltzrubi}.

It is partially these observed instabilities that motivate our current study. One of the simplest configurations for which unstable and complex flow behavior and patterns are observed is when the boundaries exhibit ion-selectivity i.e. many surfaces (membranes) that arise in biology, chemistry, and electrochemistry allow for penetration of certain ions while blocking others \cite{davidson,rubinstein}. Mathematically, this situation can be modelled by mixed boundary conditions wherein, for example, $c_1$ has selective boundary conditions and $c_2$ has blocking boundary conditions. In considering these mixed boundary conditions in three dimensions, the main mathematical difficulties include
\begin{enumerate}
    \item nonlinear, nonlocal boundary conditions (blocking)
    \item supercriticality of the nonlinear, nonlocal flux terms, $\div(c_i\na\Phi)$
    \item lack of natural dissipative structure.
\end{enumerate}
We compare our situation with related works \cite{ci}, \cite{bothe}, and \cite{cil}. In \cite{ci} and \cite{bothe} the two dimensional setting allowed for control of the nonlinear term $\div(c_i\na\Phi)$ in a large variety of situations, including blocking boundary conditions for $c_i$ and Dirichlet \cite{ci} and Robin \cite{bothe} boundary conditions for $\Phi$. In \cite{ci}, 2D global regularity is shown for mixed boundary conditions, too. However, many important steps of the analysis do not carry over to three dimensions due to the difference in scaling. This is what we mean when we say that $\div(c_i\na\Phi)$ is supercritical in three dimensions. On the other hand, in \cite{cil}, the three dimensional setting is considered and global regularity is established when Dirichlet boundary conditions are prescribed for $\Phi$ and also for all $c_i$. The issue of the supercriticality of the nonlinearity is circumvented by, simply put, transferring all the potentially ``bad" nonlinear behavior to the boundary where the behavior is a priori controlled due to the boundary conditions. An important ingredient of the analysis is that the boundary conditions give control of \textit{both} $c_1$ and $c_2$ on the boundary. This is not the case for blocking or mixed boundary conditions.

Our current work culminates in Theorem \ref{gr!!} in Section \ref{mbc}, where we consider the NPNS and NPS systems for two oppositely charged ionic species with mixed boundary conditions for $c_i$ (selective for $c_1$ and blocking for $c_2$) and Robin boundary conditions for $\Phi$. We prove large data global regularity, unconditionally for NPS and conditional on the regularity of the fluid velocity $u$ for NPNS. The general strategy is similar to \cite{cil} in that we transfer all the harmful nonlinearities to the boundary. However, the main difference is that in the mixed boundary conditions scenario, the boundary behavior is not a priori controlled the same way as it is in \cite{cil}. Thus, a careful analysis is necessary to show that the internal dissipative ``forces" of the system are strong enough to counteract the potentially problematic boundary behavior. A novel ingredient used at this stage is control of the quantity $\|c_1(t)\|_{L^1(\Omega)}$. Aside from this control, the Robin boundary conditions for $\Phi$ play an important role - while they do not prescribe the values of $\Phi$ or of $\pa_n\Phi$ on the boundary, they do weaken the nonlinearity at the boundary just enough so that dissipation dominates. A close inspection of the proof reveals that replacing the Robin boundary conditions on $\Phi$ with Dirichlet boundary conditions (as in \cite{ci,np3d,cil}) does not allow for the same proof to go through. Thus the problem of global regularity for blocking and mixed boundary conditions for $c_i$ with Dirichlet boundary conditions for $\Phi$ is, in general, open for three dimensions. On the other hand, the proof also reveals how much the analysis can be simplified if Neumann boundary conditions were chosen for $\Phi$ (see e.g. \cite{schmuck}). Robin boundary conditions are situated appropriately in between Dirichlet and Neumann boundary conditions in such a way that they still allow us to take into account applied electric potentials on the boundary (a feature that makes Robin and Dirichlet boundary conditions appealing for the study of the aforementioned electrokinetic instabilities), while retaining some of the mathematically simplifying features of Neumann boundary conditions. {Ultimately, taking Robin boundary conditions to be a physically suitable description of the electrical field at the boundary, Theorem \ref{gr!!} reveals that, in the physically relevant case of three spatial dimensions and assuming sufficient regularity of the fluid velocity field $u$, the NPNS equations do not admit solutions that blow up in finite time (e.g. Dirac mass type aggregation of ions in finite time), and solutions in fact remain regular for all positive time.}

Leading up to Section \ref{mbc}, in Sections \ref{GR} and \ref{mss}, we consider, respectively, a two species and multiple species setting where all the $c_i$ satisfy blocking boundary conditions {and prove global regularity of solutions, unconditionally for NPS and conditional on the regularity of the fluid velocity $u$ for NPNS. In the latter setting of multiple species, }we require additionally that the diffusivities and magnitudes of ionic valences are equal ($D_1=...=D_m$, $|z_1|=...=|z_m|$) (see also \cite{cil}).  Because all the $c_i$ satisfy blocking boundary conditions, there is a natural dissipative structure (Section \ref{de}), which facilitates the analysis, but unlike in the two dimensional case \cite{bothe}, this dissipation alone seems insufficient to prove global regularity. Rather the dissipation must be supplemented by precise control of the boundary behavior using the Robin boundary conditions on $\Phi$. {On one hand, we consider these cases of \textit{only} blocking boundary conditions for $c_i$ because 
energy estimates in these two sections are more concise relative to the mixed boundary conditions case, and they more clearly illustrate the role played by the Robin boundary conditions and the subsequent estimates of boundary terms. On the other hand, the results of these two sections are nontrivial in their own right and also serve to verify that, despite the impenetrable nature of the boundaries (modelled by blocking boundary conditions), if the fluid velocity field remains regular, then blow up of ions near the boundary (or anywhere in the domain) cannot occur in finite time, regardless of the size of the prescribed data for the electrical potential $\Phi$.}

Prior to the proofs of the main theorems, in Section \ref{prelim}, we introduce the relevant function spaces and and state a local existence theorem, the proof of which we omit but can be found in the references provided.

\section{Preliminaries}\la{prelim}
We are concerned with global existence of strong solutions of NPNS and NPS. To define what we mean by a strong solution, we first introduce the relevant function spaces.

We denote by $L^p(\Omega)=L^p$ and $W^{m,p}(\Omega)=W^{m,p}$ the standard Lebesgue spaces and Sobolev spaces, respectively. In the case $p=2$, we denote $W^{m,2}=H^m(\Omega)=H^m$. We also consider Lebesgue spaces on the boundary $\pa\Omega$: $L^p(\pa\Omega)$. In this latter case, we always indicate the underlying domain $\pa\Omega$ to avoid ambiguity. We also denote $L^p_tL^q_x=L^p(0,T;L^q(\Omega))$, $L^p_tW^{m,q}_x=L^p(0,T;W^{m,q}(\Omega)),$ where the time $T$ is made clear from context.

Denoting $\mathcal{V}=\{f\in (C_c^\infty(\Omega))^3\,|\, \div f =0\}$, the spaces $H\subset (L^2(\Omega))^3$ and $V\subset (H^1(\Omega))^3$ are the closures of $\mathcal{V}$ in $(L^2(\Omega))^3$ and $(H^1(\Omega))^3$, respectively. The space $V$ is endowed with the Dirichlet norm $\|f\|_V^2=\int_\Omega |\na f|^2\,dx$.

In order to avoid having to explicitly deal with the pressure term in the Navier-Stokes and Stokes equations, we sometimes work with the equations projected onto the space of divergence free vector fields via the Leray projector $\mathbb{P}:L^2(\Omega)^3\to H$,
\begin{align}
    \pa_t u+B(u,u)+\nu Au=-K\mathbb{P}(\rho\na\Phi)\la{nse'}\\
    \pa_t u+\nu Au=-K\mathbb{P}(\rho\na\Phi)\la{stokes'}
\end{align}
where $A=-\mathbb{P}\D:\mathcal{D}(A)\to H$, $\mathcal{D}(A)=H^2(\Omega)^3\cap V$ is the Stokes operator, and $B(u,u)=\mathbb{P}(u\cdot\na u)$ (see \cite{cf} for related theory).

\begin{defi}
We say that $(c_i,\Phi,u)$ is a strong solution of NPNS (\ref{np}),(\ref{pois}),(\ref{nse}) or of NPS (\ref{np}),(\ref{pois}),(\ref{stokes}) with boundary conditions (\ref{bl}) (or (\ref{DI})),(\ref{noslip}),(\ref{robin}) on the time interval $[0,T]$ if $c_i\ge 0$, $c_i\in L^\infty(0,T;H^1)\cap L^2(0,T;H^2)$, $u\in L^\infty(0,T;V)\cap L^2(0,T;\mathcal{D}(A))$ and $(c_i,\Phi,u)$ solve the equations in the sense of distributions and satisfy the boundary conditions in the sense of traces.\la{strong}
\end{defi}

The NPNS/NPS system is a semilinear parabolic system and local existence and uniqueness of strong solutions have been established by many authors for many different sets of boundary conditions. We refer the reader to \cite{bothe} where local well-posedness is established for dimensions greater than one, arbitrarily many ionic species, blocking boundary conditions for $c_i$, and Robin boundary conditions for $\Phi$. However, as the authors remark, the proof, based on methods of maximal $L^p$ regularity, can be adapted in a straightforward manner for different boundary conditions, including the mixed boundary conditions consiidered later in Section \ref{mbc}. Thus we have the following local existence theorem:

\begin{thm}
For initial conditions $0\le c_i(0)\in H^1$, $u(0)\in V$, there exists $T_0>0$ depending on $\|c_i(0)\|_{H^1},\|u(0)\|_V$, the boundary data $\tau,\xi$ (and $\gamma_i$ if (\ref{DI}) is considered), the parameters of the problem $D_i,z_i,\epsilon,\nu,K$, and the domain $\Omega$ such that NPNS (\ref{np}),(\ref{pois}),(\ref{nse}) (and NPS (\ref{np}),(\ref{pois}),(\ref{stokes})) has a unique strong solution $(c_i,\Phi,u)$ on the time interval $[0,T_0]$ satisfying the boundary conditions (\ref{bl}) (or (\ref{DI})),(\ref{noslip}),(\ref{robin}).\la{local}
\end{thm}

\begin{rem}
We stress that the nonnegativity of $c_i$ is included in our definition of a strong solution. That nonnegativity is propagated from nonnegative initial conditions $c_i(0)\ge 0$ is not self-evident. Its proof is included in Appendix \ref{pc}. In fact, as proved in Appendix \ref{pc}, more is true: strict positivity is propagated from strictly positive initial conditions $c_i(0)\ge c>0$.
\end{rem}

Henceforth, all occurrences of the constant $C>0$, with no subscript, refer to a constant depending only on the parameters of the system, the boundary data, and the domain $\Omega$, and this constant may differ from line to line. For brevity, when a constant, other than $C$, is said to depend on the parameters of the system, we mean this to also include boundary data and the domain.

\section{Global Regularity for Blocking Boundary Conditions (Two Species)}\la{GR}
In this section we consider the NPNS and NPS systems for two ($m=2$) oppositely charged ($z_1>0>z_2$) ionic species, both satisfying blocking boundary conditions. In this setting, we prove global existence of strong solutions for the NPS system and the same result, conditional on Navier-Stokes regularity, for the NPNS system.

We begin by proving some a priori bounds that are used for the proof of the global regularity result of this section. We prove some of the estimates in more generality (two or more species) so as to invoke them in Section \ref{mss} as well. In Sections \ref{mss} and \ref{mbc}, for the sake of brevity and fewer repetitive computations, we shall also frequently make references to some estimates from this section that may not hold verbatim but nonetheless hold up to some minor modifications.
\subsection{Dissipation Estimate}\la{de}
The NPNS and NPS systems come with a dissipative structure when blocking boundary conditions for $c_i$ and Robin boundary conditions for $\Phi$ are considered. We prove the following proposition:
\begin{prop}
Let $(c_i,\Phi,u)$ be a strong solution of NPNS or NPS on the time interval $[0,T]$, satisfying boundary conditions (\ref{bl}),(\ref{noslip}),(\ref{robin}). Then the functional
\be
V(t)=\fr{1}{2K}\|u\|_H^2+\sum_{i=1}^m\int_\Omega c_i\log c_i\,dx+\fr{\epsilon}{2}\|\na\Phi\|_{L^2(\Omega)}^2+\fr{\epsilon\tau}{2}\|\Phi\|_{L^2(\pa\Omega)}^2
\ee
satisfies
\be
\fr{d}{dt}V+\mathcal{D}+\fr{\nu}{K}\|\na u\|_{L^2}^2=0
\ee
for $t\in[0,T]$, where
\be
\mathcal{D}=\sum_{i=1}^m D_i\int_\Omega c_i|\na\mu_i|^2\,dx\ge 0
\ee
and $\mu_i$ is the electrochemical potential,
\be
\mu_i=\log c_i+z_i\Phi.\la{mu}
\ee
In particular, $V(t)$ is nonincreasing in time.\la{diss}
\end{prop}

\begin{proof}
First we note that (\ref{np}) can be written
\be
\pa_t c_i+u\cdot \na c_i=D_i\div(c_i\na\mu_i).\la{np'}
\ee
Then we multiply (\ref{np'}) by $\mu_i$ and integrate by parts. This part is somewhat formal as we cannot exclude the possibility that $c_i$ attains the value $0$, in which case $\log c_i$ becomes undefined. Thus, to make this rigorous, we can work instead with the quantity $\mu_i^\delta=\log(c_i+\delta)+z_i\Phi$ and later pass to the limit $\delta\to 0$, as done in \cite{bothe}. For conciseness, we stick to the formal computations involving $\mu_i$. On the right hand side of (\ref{np'}), we obtain after summing in $i$,
\be
-\sum_{i=1}^m D_i\int_\Omega c_i|\na\mu_i|^2\,dx.
\ee
which is precisely $-\mathcal{D}$. For the terms on the left hand side, we have after summing in $i$ and integrating by parts,
\be
\bal
\sum_{i=1}^m\int_\Omega \pa_t c_i(\log c_i+z_i\Phi)\,dx&=\fr{d}{dt}\sum_{i=1}^m\int_\Omega c_i\log c_i-c_i\,dx+\int_\Omega(\pa_t\rho)\Phi\,dx\\
&=\fr{d}{dt}\sum_{i=1}^m\int_\Omega c_i\log c_i\,dx-\epsilon\int_\Omega\pa_t(\D\Phi)\Phi\,dx\\
&=\fr{d}{dt}\sum_{i=1}^m\int_\Omega c_i\log c_i\,dx-\epsilon\int_{\pa\Omega}\pa_t(\pa_n\Phi)\Phi\,dS+\fr{\epsilon}{2}\fr{d}{dt}\int_\Omega|\na\Phi|^2\,dx\\
&=\fr{d}{dt}\sum_{i=1}^m\int_\Omega c_i\log c_i\,dx+\epsilon\tau\int_{\pa\Omega}(\pa_t\Phi)\Phi\,dS+\fr{\epsilon}{2}\fr{d}{dt}\int_\Omega|\na\Phi|^2\,dx\\
&=\fr{d}{dt}\sum_{i=1}^m\int_\Omega c_i\log c_i\,dx+\fr{\epsilon\tau}{2}\fr{d}{dt}\int_{\pa\Omega}\Phi^2\,dS+\fr{\epsilon}{2}\fr{d}{dt}\int_\Omega|\na\Phi|^2\,dx.
\eal
\ee
In the second line, we used the fact that $\|c_i(t)\|_{L^1}=\|c_i(0)\|_{L^1}$ for all time due to blocking boundary conditions. We also used the Poisson equation for $\Phi$. In the fourth line, we used the Robin boundary conditions (\ref{robin}).

Lastly, for the advective term we obtain after summing,
\be
\bal
\sum_{i=1}^m\int_\Omega u\cdot\na c_i(\log c_i+z_i\Phi)\,dx=&\sum_{i=1}^m\int_\Omega u\cdot\na(c_i\log c_i-c_i)\,dx+\int_\Omega (u\cdot\na\rho)\Phi\,dx\\
=&\int_\Omega (u\cdot\na\rho)\Phi\,dx\\
=&-\int_\Omega (u\cdot\na\Phi)\rho\,dx
\eal
\ee
where in the second and third lines we integrated by parts and used the fact that $\div u=0$. Collecting what we have so far, we have
\be
\fr{d}{dt}\left(\sum_{i=1}^m\int_\Omega c_i\log c_i\,dx+\fr{\epsilon}{2}\|\na\Phi\|_{L^2}^2+\fr{\epsilon\tau}{2}\|\Phi\|_{L^2(\pa\Omega)}^2\right)+\mathcal{D}=\int_\Omega(u\cdot\na\Phi)\rho\,dx.\la{one}
\ee
Next we multiply (\ref{nse}) (or (\ref{stokes})) by $u$ and integrate by parts, noticing that the integral corresponding to the nonlinear term for Navier-Stokes vanishes due to the divergence-free condition,
\be
\fr{1}{2}\fr{d}{dt}\|u\|_{L^2}^2+\nu\|\na u\|_{L^2}^2=-K\int_\Omega (u\cdot\na\Phi)\rho\,dx.\la{two}
\ee
Thus, multiplying (\ref{two}) by $K^{-1}$ and adding it to (\ref{one}), we obtain the conclusion of the proposition.
\end{proof}

\subsection{Uniform $L^2$ Estimate}\la{UL2}
Using the dissipative estimate from the previous subsection, we obtain uniform in time control of $\|c_i\|_{L^2}$  in the case of two oppositely charged species satisfying blocking boundary conditions.

\begin{prop}
Let $(c_i,\Phi,u)$ be a strong solution of NPNS or NPS for two oppositely charged species ($m=2,\, z_1>0>z_2$) on the time interval $[0,T]$, satisfying boundary conditions (\ref{bl}),(\ref{noslip}),(\ref{robin}), and corresponding to initial conditions $0\le c_i(0)\in H^1,\,u(0)\in V$. Then there exists a constant $M_2>0$ independent of time, depending only on the parameters of the system and the initial conditions such that for each $i$
\be
\sup_{t\in[0,T]}\|c_i(t)\|_{L^2}<M_2.\la{m2m2}
\ee\la{L2'}
\end{prop}
\begin{proof}
We multiply (\ref{np}) by $\fr{|z_i|}{D_i}c_i$ and integrate by parts:
\be
\bal
\fr{|z_i|}{2D_i}\fr{d}{dt}\int_\Omega c_i^2\,dx=&-|z_i|\int_\Omega|\na c_i|^2\,dx-z_i|z_i|\int_\Omega c_i\na c_i\cdot\na\Phi\,dx\\
=&-|z_i|\int_\Omega|\na c_i|^2\,dx-z_i|z_i|\fr{1}{2}\int_\Omega \na c_i^2\cdot\na\Phi\,dx\\
=&-|z_i|\int_\Omega|\na c_i|^2\,dx-z_i|z_i|\fr{1}{2\epsilon}\int_\Omega c_i^2\rho\,dx- z_i|z_i|\fr{1}{2}\int_{\pa\Omega}c_i^2\pa_n\Phi\,dS\\
=&-|z_i|\int_\Omega|\na c_i|^2\,dx-z_i|z_i|\fr{1}{2\epsilon}\int_\Omega c_i^2\rho\,dx\\
&+z_i|z_i|\fr{\tau}{2}\int_{\pa\Omega}c_i^2\Phi\,dS- z_i|z_i|\fr{1}{2}\int_{\pa\Omega}c_i^2\xi\,dS\\
=&-|z_i|\int_\Omega|\na c_i|^2\,dx-z_i|z_i|\fr{1}{2\epsilon}\int_\Omega c_i^2\rho\,dx+I_1^{(i)}+I_2^{(i)}\la{ener}
\eal
\ee
where in the fourth line, we used the Robin boundary conditions (\ref{robin}). We estimate the two boundary integrals using trace inequalities (Lemma \ref{trace}, Appendix):
\be
\bal
|I_1^{(i)}|\le&C\|\Phi\|_{L^4(\pa\Omega)}\|c_i\|_{L^\fr{8}{3}(\pa\Omega)}^2\\
\le&\|\Phi\|_{H^1(\Omega)}(C_\delta\|c_i\|_{L^1(\Omega)}^2+\delta\|\na c_i\|_{L^2(\Omega)}^2)\la{I_1}
\eal
\ee
and similarly,
\be
\bal
|I_2^{(i)}|\le&\fr{|z_i|^2\|\xi\|_{L^\infty(\pa\Omega)}}{2}\|c_i\|_{L^2(\pa\Omega)}^2\le C_\delta\|c_i\|_{L^1}^2+\delta\|\na c_i\|_{L^2}^2.\la{I_2}
\eal
\ee
We recall that $\|c_i\|_{L^1}$ remains constant in time, and since by a generalized Poincaré's inequality we have that \be
\|\Phi\|_{L^2(\Omega)}\le C(\|\na\Phi\|_{L^2(\Omega)}+\|\Phi\|_{L^2(\pa\Omega)})
\ee
we deduce from Proposition \ref{diss} that $\|\Phi\|_{H^1}$ is uniformly bounded in time by initial data. Thus choosing 
\be
\delta=\min\left\{\fr{|z_i|}{4},\fr{|z_i|}{4\sup_t\|\Phi(t)\|_{H^1}}\right\}
\ee
we obtain from (\ref{ener})-(\ref{I_2}), after summing in $i$ and recalling $z_1>0>z_2$,
\be
\bal
\sum_{i=1}^2\fr{|z_i|}{2D_i}\fr{d}{dt}\|c_i\|_{L^2(\Omega)}^2+\sum_{i=1}^2\fr{|z_i|}{2}\|\na c_i\|_{L^2(\Omega)}^2&\le C_b-\fr{1}{2\epsilon}\int_\Omega (z_1^2c_1^2-z_2^2c_2^2)\rho\,dx\\
&= C_b-\fr{1}{2\epsilon}\int_\Omega \rho^2(|z_1|c_1+|z_2|c_2)\,dx\\
&\le C_b.\la{cancellation}
\eal
\ee
Here, $C_b$ depends on $\sup_t\|\Phi(t)\|_{H^1}, \|c_i(0)\|_{L^1}$ along with the various parameters of the system. Next we use a Gagliardo-Nirenberg inequality to bound 
\be
\|c_i\|_{L^2}^2\le C(\|\na c_i\|_{L^2}^2+\|c_i\|_{L^1}^2)\le C'(\|\na c_i\|_{L^2}^2+1),\la{poinc}
\ee
where $C'$ depends on $\|c_i(0)\|_{L^1}$. Thus, for constants $C'_b, C''_b\ge 0$ depending on $\sup_t\|\Phi(t)\|_{H^1}, \|c_i(0)\|_{L^1}$ and parameters, we obtain from (\ref{cancellation}),
\be
\fr{d}{dt}\left(\sum_{i=1}^2\fr{|z_i|}{2D_i}\|c_i\|_{L^2}^2\right)\le -C'_b\left(\sum_{i=1}^2\fr{|z_i|}{2D_i}\|c_i\|_{L^2}^2\right)+C''_b.\la{last}
\ee
Thus from a Grönwall estimate, we find
\be
\sum_{i=1}^2\fr{|z_i|}{2D_i}\|c_i(t)\|_{L^2}^2\le \left(\sum_{i=1}^2\fr{|z_i|}{2D_i}\|c_i(0)\|_{L^2}^2\right)e^{-C'_bt}+\fr{C''_b}{C'_b}(1-e^{-C'_bt})\le \sum_{i=1}^2\fr{|z_i|}{2D_i}\|c_i(0)\|_{L^2}^2+\fr{C''_b}{C'_b}
\ee
and (\ref{m2m2}) follows.\end{proof}

\subsection{Higher Order Estimates}\la{he} Now we bootstrap the dissipative and $L^2$ estimates to obtain some higher order estimates.
\begin{prop}
Let $(c_i,\Phi,u)$ be a strong solution of NPNS or NPS on the time interval $[0,T]$ with $0\le c_i(0)\in H^1\cap L^\infty,\,u(0)\in V$, satisfying boundary conditions (\ref{bl}),(\ref{noslip}),(\ref{robin}). Assume that for each $i$, $c_i$ satisfies a uniform in time $L^2$ bound,
\be
\|c_i(t)\|_{L^2}<M_2.\la{L2'''}
\ee
Then there exist constants $M_\infty,M'_2>0$ independent of time, depending only on the parameters of the system, the initial conditions and $M_2$ such that for each $i$
\begin{align}
\sup_{t\in[0,T]}\|c_i(t)\|_{L^\infty}&<M_\infty\la{Mi}\\
\int_{0}^{T}\|\na \tci(s)\|_{L^2}^2\,ds&< M'_2\la{L2''}
\end{align}
where $\tci=c_ie^{z_i\Phi}$. Specifically for the case of the NPS system, we have
\be
\sup_{t\in[0,T]}\|u(t)\|_V^2+\fr{1}{T}\int_0^T\|u(s)\|_{H^2}^2\,ds<B\la{Mu}
\ee
for a constant $B$ independent of time. For the NPNS system, we have instead
\be
\sup_{t\in[0,T]}\|u(t)\|_V^2+\int_0^T\|u(s)\|_{H^2}^2\,ds<B_T\la{M'u}
\ee
for a time dependent constant $B_T$ depending also on $U(T)$ where
\be
U(T)=\int_0^T\|u(s)\|_V^4\,ds.
\ee\la{L2}
\end{prop}
\begin{rem}
For two oppositely charged species, $m=2$, $z_1>0>z_2$, the hypothesis (\ref{L2'''}) holds due to Proposition \ref{L2'}.
\end{rem}
\begin{rem}
Here, and in later theorems, the assumption that $c_i(0)\in L^\infty$ is not, strictly speaking, necessary as the local existence theorem guarantees that $H^1$ initial data is immediately regularized so that on the interval of existence $[0,T]$, we have $c_i\in L^2(0,T;H^2)$. In particular, for some arbitrarily small $\tilde t>0$ we have $c_i(\tilde t)\in H^2\subset L^\infty$. Thus, below, when we derive a priori upper bounds in terms of $\|c_i(0)\|_{L^\infty}$ (c.f. (\ref{sk})), we could instead do so in terms of $\|c_i(\tilde t)\|_{L^\infty}$. To avoid doing this, we include $c_i(0)\in L^\infty$ in the hypothesis. For later theorems (e.g. in Section \ref{mbc}) whose proofs do not invoke $\|c_i(0)\|_{L^\infty},$ we do note include $c_i\in L^\infty$ in the hypothesis.
\end{rem}
\begin{proof}
It follows from (\ref{pois}), (\ref{L2'''}) and Sobolev estimates that for some constants $P_6,\,p_\infty$ independent of time, we have,
\begin{align}
\|\na\Phi(t)\|_{L^6}&\le P_6\la{P6}\\
\|\Phi(t)\|_{L^\infty}&\le p_\infty\la{Piii}.
\end{align}

Now we deduce the uniform in time $L^\infty$ bounds, using a Moser-type iteration (see also \cite{bothe,choi,np3d,schmuck}). For $k=2,3,4...$, we multiply (\ref{np}) by $c_i^{2k-1}$ and integrate by parts to obtain
\be
\fr{1}{2k}\fr{d}{dt}\|c_i\|_{L^{2k}}^{2k}+\fr{2k-1}{k^2}D_i\|\na c_i^k\|_{L^2}^2\le C\fr{2k-1}{k}\|\na\Phi\|_{L^6}\|c_i^k\|_{L^3}\|\na c_i^k\|_{L^2}.\la{kk}
\ee
We use (\ref{P6}) and interpolate $L^3$ between $L^1$ and $H^1$ to obtain after a Young's inequality,
\be
\fr{d}{dt}\|c_i\|_{L^{2k}}^{2k}+D_i\|\na c_i^k\|_{L^2}^2\le C_k\|c_i^k\|_{L^1}^2\la{2k}
\ee
where $C_k$ satisfies, for some $c>0$, for some $m$ large enough and for each $k=2,3,4...$
\be
C_k\le ck^m.\la{ck}
\ee
Interpolating $L^2$ between $L^1$ and $H^1$,
\be
\|c_i^k\|_{L^2}^2\le C(\|\na c_i^k\|_{L^2}^2+\|c_i^k\|_{L^1}^2),
\ee
we obtain from (\ref{2k}),
\be
\fr{d}{dt}\|c_i\|_{L^{2k}}^{2k}\le -C\|c_i\|_{L^{2k}}^{2k}+C_k\|c_i^k\|_{L^1}^2=-C\|c_i\|_{L^{2k}}^{2k}+C_k\|c_i\|_{L^{k}}^{2k}\la{2k'}
\ee
for a different $C_k$ still satisfying $(\ref{ck})$ for some $c$. We define
\be
S_k=\max\{\|c_i(0)\|_{L^\infty},\,\sup_t\|c_i(t)\|_k\}.\la{sk}
\ee
Applying a Grönwall inequality to (\ref{2k'}), we obtain
\be
\|c_i\|_{L^{2k}}^{2k}\le\|c_i(0)\|_{L^{2k}}^{2k}+C_kS_k^{2k}\le|\Omega|\|c_i(0)\|_{L^\infty}^{2k}+C_kS_k^{2k}\le Ck^mS_k^{2k}\la{lel}
\ee
for a possibly different $C_k$ still satisfying $(\ref{ck})$ for some $c$.  Assuming without loss of generality that $C\ge 1$ in (\ref{lel}),
\be
S_{2k}=\max\{\|c_i(0)\|_{L^\infty},\sup_t\|c_i(t)\|_{L^{2k}}\}\le\max\{\|c_i(0)\|_{L^\infty},C^{\fr{1}{2k}}k^{\fr{m}{2k}}S_k\}=C^{\fr{1}{2k}}k^{\fr{m}{2k}}S_k.\la{ss}
\ee
Setting $k=2^j$, we obtain
\be
S_{2^{j+1}}\le C^{\fr{1}{2^{j+1}}}2^{\fr{jm}{2^{j+1}}}S_{2^j}
\ee
and thus for all $J\in\mathbb{N}$
\be
S_{2^J}\le C^a2^bS_2<\infty\la{s2j}
\ee
where
\be
a=\sum_{j=1}^\infty \fr{1}{2^{j+1}}<\infty,\quad b=\sum_{j=1}^\infty \fr{jm}{2^{j+1}}<\infty.
\ee
Passing $J\to\infty$ in (\ref{s2j}), we obtain (\ref{Mi}).

Next, (\ref{L2''}) follows from (\ref{Mi}), (\ref{Piii}) and Proposition \ref{diss}. Indeed, from the proposition, using the fact that $\mu_i=\log \tci$, we obtain
\be
\int_0^T\int_\Omega|\na \tci|^2\,dx\,dt\le C_p\int_0^T\int_\Omega c_i|\na\mu_i|^2\,dx\,dt\le M'_2
\ee
for $M'_2$ independent of time, and $C_p$ depending on $p_\infty$ and $M_\infty$. 

Next, to prove (\ref{Mu}), we multiply the Stokes equations (\ref{stokes'}) by $Au$ and integrate by parts,
\be
\bal
\fr{1}{2}\fr{d}{dt}\|u\|_V^2+\fr{\nu}{2}\|A u\|_{L^2}^2\le C\|\rho\na\Phi\|_{L^2}^2\le C'M_\infty^2\la{nn}
\eal
\ee
where $C'$ depends on $\sup_t\|\na\Phi\|_{L^2}$ (Proposition \ref{diss}). Using the elliptic bound $\|u\|_V\le C\|A u\|_{L^2}$, we obtain
\be
\fr{1}{2}\fr{d}{dt}\|u\|_V^2\le -C''\|u\|_V^2+C'M_\infty^2 \la{unps}
\ee
which gives us uniform boundedness of $\|u\|_V$, the first half of (\ref{Mu}). Integrating (\ref{nn}) in time gives us the second half.

Similarly, for NPNS, we multiply the Navier-Stokes equations (\ref{nse'}) by $A u$ and integrate by parts,
\be
\bal
\fr{1}{2}\fr{d}{dt}\|u\|_V^2+\fr{\nu}{2}\|A u\|_{L^2}^2\le& C'M_\infty^2+\|u\|_{L^6}\|\na u\|_{L^3}\|A u\|_{L^2}\\
\le& C'M_\infty^2 +C\|u\|_V^\fr{3}{2}\|A u\|_{L^2}^\fr{3}{2}
\eal
\ee
so that after a Young's inequality we obtain
\be
\bal
\fr{1}{2}\fr{d}{dt}\|u\|_V^2+\fr{\nu}{4}\|A u\|_{L^2}^2\le& C'M_\infty^2+C\|u\|_V^6\la{nnn}
\eal
\ee
from which we obtain
\be
\fr{1}{2}\|u(t)\|_{V}^2+\fr{\nu}{4}\int_0^t\|A u(s)\|_{L^2}^2\,ds\le\left(\fr{1}{2}\|u(0)\|_V^2+C'M_\infty^2t\right)e^{CU(T)}.
\ee
This gives us (\ref{M'u}) and completes the proof of the proposition.
\end{proof}

\subsection{Proof of Global Regularity}\la{gr}
Now we prove our main global regularity result of this section.
\begin{thm}
For initial conditions $0\le c_i(0)\in H^1$, $u(0)\in V$ and for all $T>0$, NPS (\ref{np}),(\ref{pois}),(\ref{stokes}) for two oppositely charged species ($m=2,\,z_1>0>z_2$) has a unique strong solution $(c_i,\Phi,u)$ on the time interval $[0,T]$ satisfying the boundary conditions (\ref{bl}),(\ref{noslip}),(\ref{robin}). NPNS (\ref{np}),(\ref{pois}),(\ref{nse}) for two oppositely charged species has a unique strong solution on $[0,T]$ satisfying the initial and boundary conditions provided
\be
U(T)=\int_0^T\|u(s)\|_V^4\,ds<\infty.
\ee
Moreover, the solution to NPS satisfies (\ref{Mu}) in addition to 
\begin{align}
\sup_{t\in[0,T]}\|c_i(t)\|_{H^1}^2+\fr{1}{T}\int_0^T\|c_i(s)\|_{H^2}^2\,ds&\le M\la{M}\\
\sup_{t\in[0,T]}\|\na\tci(t)\|_{L^2}^2+\int_0^T\|\D\tci(s) \|_{L^2}^2\,ds&\le M'\la{M'}
\end{align}
for $\tci=e^{z_i\Phi}$ and constants $M,M'$ depending only on the parameters of the system and the initial conditions but not on $T$. The solution to NPNS satisfies (\ref{M'u}) in addition to
\be
\sup_{t\in[0,T]}\|c_i(t)\|_{H^1}^2+\int_0^T\|c_i(s)\|_{H^2}^2\,ds\le M_T\la{MU}
\ee
for a constant $M_T$ depending on the parameters of the system, the initial conditions, $T$, and $U(T)$.\la{gr!}
\end{thm}

\begin{proof}
We prove the a priori estimates (\ref{M})-(\ref{MU}), which together with the bounds on $u$, (\ref{Mu}) and (\ref{M'u}), and the local existence theorem allow us to uniquely extend a local solution to a global one by virtue of the fact that the strong norms in Definition \ref{strong} do not blow up before time $T$.

Due to (\ref{Mi}), it follows from the embedding $W^{2,\infty}\hookrightarrow W^{1,\infty}$ that
\be
\|\Phi(t)\|_{W^{1,\infty}}\le P_\infty\la{Pi}
\ee
for all $t$, where $P_\infty$ depends only on the parameters of the system and uniform $L^p$ bounds on $\rho$.

Next, in order to obtain estimates for $\na c_i$, we note that the auxiliary variable
\be
\tci=c_ie^{z_i\Phi}\la{cit}
\ee
satisfies
\be
\pa_t\tci+u\cdot\na\tci=D_i\D\tci-D_iz_i\na\tci\cdot\na\Phi+z_i((\pa_t+u\cdot\na)\Phi)\tci\la{npt}
\ee
together with homogeneous Neumann boundary conditions
\be
{\pa_n\tci}_{|\pa\Omega}=0.\la{neu}
\ee
Multiplying (\ref{npt}) by $-\D\tci$ and using (\ref{neu}) to integrate by parts, we obtain
\be
\bal
\fr{1}{2}\fr{d}{dt}\|\na\tci\|_{L^2}^2+D_i\|\D\tci\|_{L^2}^2=&\int_\Omega (u\cdot\na\tci)\D\tci\,dx+D_iz_i\int_\Omega(\na\tci\cdot\na\Phi)\D\tci\,dx\\
&-z_i\int_\Omega((\pa_t+u\cdot\na)\Phi)\tci\D\tci\,dx\\
=& I_1+I_2+I_3.\la{dci}
\eal
\ee
We estimate using Hölder and Young's inequalities and Sobolev and interpolation estimates,
\be
\begin{aligned}
|I_1|\le&\|u\|_V\|\na\tci\|_{L^3}\|\D\tci\|_{L^2}\\
\le& C\|u\|_V\|\na\tci\|_{L^2}^\fr{1}{2}\|\D \tci\|_{L^2}^\fr{3}{2}\\
\le&\fr{D_i}{4}\|\D \tci\|_{L^2}^2+C\|u\|_V^4\|\na\tci\|_{L^2}^2\la{id}
\end{aligned}
\ee

\be
\bal
|I_2|\le&C\|\na\Phi\|_{L^\infty}\|\na \tci\|_{L^2}\|\D\tci\|_{L^2}\\
\le&\fr{D_i}{4}\|\D\tci\|_{L^2}^2+C_g\|\na\tci\|_{L^2}^2\la{idd}
\eal
\ee
where $C_g$ depends on $P_\infty$.
Next we split
\be
I_3=-z_i\int_\Omega(u\cdot\na\Phi)\tci\D\tci\,dx-z_i\int_\Omega(\pa_t\Phi)\tci\D\tci\,dx=I_3^1+I_3^2.\la{split}
\ee
First we estimate $I_3^1$. Noting that
\be
\|\tci\|_{L^3}\le CM_\infty e^{|z_i|P_\infty}=\beta_3
\ee
we bound
\be
\bal
|I_3^1|\le& C\|u\|_V\|\na\Phi\|_{L^\infty}\|\tci\|_{L^3}\|\D \tci\|_{L^2}\\
\le&\fr{D_i}{8}\|\D\tci\|_{L^2}^2+C'_g\|u\|_V^2
\eal
\ee
where $C'_g$ depends on $\beta_3$ and $P_\infty$.

In order to bound $I_3^2$, first we note that the Nernst-Planck equations (\ref{np}) can be written
\be
\pa_tc_i+u\cdot\na c_i=D_i\div(e^{-z_i\Phi}\na\tci)
\ee
so that in particular we have
\be
\pa_t\rho=\sum_{i=1}^2z_iD_i\div(e^{-z_i\Phi}\na \tci)-u\cdot\na\rho.\la{ptrho}
\ee
We multiply (\ref{ptrho}) by $\epsilon^{-1}\pa_t\Phi$ and integrate by parts. On the left hand side, we have
\be
\bal
\int_\Omega \pa_t\rho(\epsilon^{-1}\pa_t\Phi)\,dx=-\int_\Omega\pa_t\D\Phi\pa_t\Phi\,dx=&-\int_{\pa\Omega}\pa_t\pa_n\Phi\pa_t\Phi\,dS+\int_\Omega|\na\pa_t\Phi|^2\,dx\\
=&\tau\int_{\pa\Omega}|\pa_t\Phi|^2\,dS+\int_\Omega|\na\pa_t\Phi|^2\,dx
\eal
\ee
where in the second line we used the Robin boundary conditions (\ref{robin}). Therefore
\be
\bal
\tau\|\pa_t\Phi\|_{L^2(\pa\Omega)}^2+\|\pa_t\na\Phi\|_{L^2(\Omega)}^2\le& C\sum_{j=1}^2\left|\int_\Omega\div(e^{-z_j\Phi}\na\tcj)\pa_t\Phi\,dx\right|+C\left|\int_\Omega\div(u\rho)\pa_t\Phi\,dx\right|\\
\le&C\sum_{j=1}^2\int_\Omega e^{|z_j|P_\infty}|\na \tcj||\pa_t\na\Phi|\,dx+C\int_\Omega M_\infty|u||\pa_t\na\Phi|\,dx\\
\le&C_{t}(\sum_{j=1}^2\|\na\tcj\|_{L^2}+\|u\|_V)\|\pa_t\na\Phi\|_{L^2}\la{ptg}
\eal
\ee
where $C_{t}$ depends on $M_\infty,P_\infty$. Therefore, we obtain
\be
(\tau
\|\pa_t\Phi\|_{L^2(\pa\Omega)}^2+\|\pa_t\na\Phi\|_{L^2(\Omega)}^2)^\fr{1}{2}\le C_{t}(\sum_{j=1}^2\|\na\tcj\|_{L^2}+\|u\|_V).\la{blaa}
\ee
Now we use a generalized Poincaré inequality, $\|f\|_{L^2(\Omega)}\le C(\|\na f\|_{L^2(\Omega)}+\|f\|_{L^2(\pa\Omega)})$, which together with Sobolev's inequality $\|f\|_{L^6}\le C\|f\|_{H^1}$ and (\ref{blaa}), gives us  
\be
\|\pa_t\Phi\|_{L^6}\le CC_{t}(\sum_{j=1}^2\|\na\tcj\|_{L^2}+\|u\|_V).\la{pt6}
\ee
Now we bound $I_3^2$ using (\ref{pt6})
\be
\bal
|I_3^2|\le& C\|\pa_t\Phi\|_{L^6}\|\tci\|_{L^3}\|\|\D\tci\|_{L^2}\\
\le&\fr{D_i}{8}\|\D\tci\|_{L^2}^2+C_3(\sum_{j=1}^2\|\na\tcj\|_{L^2}^2+\|u\|_V^2)\la{I32}
\eal
\ee
where $C_3$ depends on $\beta_3$.

Thus, adding the estimates for $I_1,I_2,I_3^1,I_3^2$ and summing in $i$, we obtain from (\ref{dci}),
\be
\fr{1}{2}\fr{d}{dt}\sum_{i=1}^2\|\na \tci\|_{L^2}^2+\fr{D_i}{4}\sum_{i=1}^2\|\D\tci\|_{L^2}^2\le C_F(\|u\|_V^4+1)\fr{1}{2}\sum_{i=1}^2\|\na \tci\|_{L^2}^2+C'_F\|u\|_V^2\la{cru}
\ee
where $C_F,C'_F$ depend on $M_\infty,P_\infty$. 

In the case of NPNS, we obtain from (\ref{cru}),
\be
\bal
\fr{1}{2}\sum_{i=1}^2&\|\na\tci(t)\|_{L^2}^2+\sum_{i=1}^2\fr{D_i}{4}\int_0^t\|\D\tci(s)\|_{L^2}^2\,ds\\
\le&\left(\fr{1}{2}\sum_{i=1}^2\|\na\tci(0)\|_{L^2}^2+C'_F\int_0^t\|u(s)\|_V^2\,ds\right)\exp\left(C_F\int_0^t\|u(s)\|_V^4+1\,ds\right).
\eal
\ee
This gives us (\ref{MU}) after converting to the original variables $c_i$ using the definition of $\tci$ and the uniform bounds on $c_i,\Phi$.

For the NPS system, we recall from Propositions \ref{diss} and \ref{L2} that $\|u\|_V$ is uniformly bounded in time and that $\|u\|_V^2$ and $\|\na \tci\|_{L^2}^2$ decay and are integrable in time, with a bound that is independent of $T$. Therefore, integrating (\ref{cru}) in time, we find that the right hand side is bounded by a constant independent of time, giving us (\ref{M'}). Converting back to the variables $c_i$ gives us (\ref{M}).
\end{proof}

\section{Global Regularity for Blocking Boundary Conditions (Multiple Species)}\la{mss} 
The results of Sections \ref{de} and \ref{he} hold for more than two species, $m>2$. Therefore the question of whether or not we can extend our global regularity result to a multiple species setting boils down to whether or not we can establish the uniform $L^2$ estimates from Section \ref{UL2} in our current setting. Once such a bound is established, the proof of global regularity for multiple species follows exactly as in the proof of Theorem \ref{gr!}.

In the proof of Proposition \ref{L2'} of Section \ref{UL2}, specifically in (\ref{cancellation}), we use the fact that, due to the assumption of two species and $z_1>0>z_2$,
\be
(z_1^2c_1^2-z_2^2c_2^2)\rho=(|z_1|c_1-|z_2|c_2)(|z_1|c_1+|z_2|c_2)\rho=\rho^2(|z_1|c_1+|z_2|c_2)\ge 0.\la{sign}
\ee
Even for the simplest extension of taking 3 species, with $z_1,z_2>0>z_3$, the corresponding leftmost term in (\ref{sign}) is $(z_1^2c_1^2+z_2^2c_2^2-z_3^2c_3^2)\rho$, and in general this term need not be nonnegative. Due to this fact, an analogous proof does not work for more than two species. However, there is one special multiple species setting where we do have global regularity: namely if all the diffusivities are equal $D_1=...=D_m$ and all the valences have the same magnitude (e.g. $z_1=z_2=1=-z_3$) (see also \cite{cil}). Indeed, we have the following theorem,

\begin{thm}
For initial conditions $0\le c_i(0)\in H^1\,(i=1,...,m)$, $u(0)\in V$ and for all $T>0$, if $D_1=...=D_m=D$ for a common value $D$ and $|z_1|=...=|z_m|=z$ for a common value $z$, then NPS (\ref{np}),(\ref{pois}),(\ref{stokes}) has a unique strong solution $(c_i,\Phi,u)$ on the time interval $[0,T]$ satisfying the boundary conditions (\ref{bl}),(\ref{noslip}),(\ref{robin}). For NPNS (\ref{np}),(\ref{pois}),(\ref{nse}), under the same hypotheses, a unique strong solution on $[0,T]$ satisfying the initial and boundary conditions exists provided
\be
\int_0^T\|u(s)\|_V^4\,ds<\infty.
\ee
Moreover, in the case of NPS, the solution satisfies the bounds (\ref{Mi}),(\ref{L2''}),(\ref{Mu}),(\ref{M}),(\ref{M'}) for each $i$. In the case of NPNS, the solution satisfies the bounds (\ref{Mi}),(\ref{L2''}),(\ref{M'u}),(\ref{MU}) for each $i$.\la{grm}
\end{thm}

\begin{proof}
We only prove the result corresponding to Proposition \ref{L2'}. As discussed at the beginning of this section, the other results leading to the proof of global regularity extend naturally from Section \ref{GR}.

This special case of multiple species effectively boils down to a two species setting. The variables $\rho$ and $\sigma=z(c_1+c_2+...+c_m)\ge 0$ satisfy
\begin{align}
\pa_t\rho+u\cdot\na\rho&=D\div(\na\rho+z\sigma\na\Phi)\la{rhoeq}\\
\pa_t\sigma+u\cdot\na\sigma&=D\div(\na\sigma+z\rho\na\Phi)\la{sigmaeq}
\end{align}
and boundary conditions
\be
\bal
(\pa_n\rho+z\sigma\pa_n\Phi)_{|\pa\Omega}=(\pa_n\sigma+z\rho\pa_n\Phi)_{|\pa\Omega}=0.
\eal
\ee                                                                                                                       Multiplying (\ref{rhoeq}) and (\ref{sigmaeq}) by $\rho$ and $\sigma$ respectively and integrating by parts, we obtain
\begin{align}
    \fr{1}{2D}\fr{d}{dt}\|\rho\|_{L^2}^2+\|\na\rho\|_{L^2}^2&=-z\int_\Omega\sigma\na\rho\cdot\na\Phi\,dx\la{bla}\\
    \fr{1}{2D}\fr{d}{dt}\|\sigma\|_{L^2}^2+\|\na\sigma\|_{L^2}^2&=-z\int_\Omega\rho\na\sigma\cdot\na\Phi\,dx\la{blabla}.
\end{align}
We estimate the right hand side of (\ref{bla}) using the boundary conditions,
\be
\bal
-z\int_\Omega\sigma\na\rho\cdot\na\Phi\,dx=&-z\int_{\pa\Omega}\sigma\rho\pa_n\Phi\,dS+z\int_\Omega\rho\na\sigma\cdot\na\Phi\,dx-\fr{z}{\epsilon}\int_\Omega\rho^2\sigma\,dx\\
\le&z\tau\int_{\pa\Omega}\sigma\rho\Phi\,dS-z\int_{\pa\Omega}\sigma\rho\xi\,dS+z\int_\Omega\rho\na\sigma\cdot\na\Phi\,dx\\
=&I_1+I_2+z\int_\Omega\rho\na\sigma\cdot\na\Phi\,dx.\la{bb}
\eal
\ee
As in (\ref{I_1}) and (\ref{I_2}), we obtain using Lemma \ref{trace} in the Appendix,
\be
\bal
    |I_1|\le& C\|\Phi\|_{L^4(\pa\Omega)}\|\rho\|_{L^\fr{8}{3}(\pa\Omega)}\|\sigma\|_{L^\fr{8}{3}(\pa\Omega)}\\
        \le&\|\Phi\|_{H^1(\Omega)}\left(C_\delta(\|\rho\|_{L^1}^2+\|\sigma\|_{L^1}^2)+\delta(\|\na\rho\|_{L^2}^2+\|\na\sigma\|_{L^2}^2)\right)
\eal
\ee
and similarly
\be
|I_2|\le C_\delta(\|\rho\|_{L^1}^2+\|\sigma\|_{L^1}^2)+\delta(\|\na\rho\|_{L^2}^2+\|\na\sigma\|_{L^2}^2).\la{I22}
\ee
Thus choosing
\be
\delta=\min\left\{\fr{1}{4},\fr{1}{4\sup_t\|\Phi(t)\|_{H^1(\Omega)}}\right\}
\ee
we obtain by adding (\ref{bla}) and (\ref{blabla}) and using the estimates (\ref{bb})-(\ref{I22}),
\be
\fr{1}{2D}\fr{d}{dt}\left(\|\rho\|_{L^2}^2+\|\sigma\|_{L^2}^2\right)+\fr{1}{2}\left(\|\na\rho\|_{L^2}^2+\|\na\sigma\|_{L^2}^2\right)<R\la{RRRR}
\ee
where $R$ is a constant that depends on uniform bounds on $\|\Phi\|_{H^1},\|\rho\|_{L^1}$ and $\|\sigma\|_{L^1}$, along with the parameters of the system. Thus by applying Gagliardo-Nirenberg inequalities to $\|\na\rho\|_{L^2}$ and $\|\na\sigma\|_{L^2}$ as in (\ref{poinc}), we conclude after a Grönwall estimate on (\ref{RRRR}) that $\|\rho\|_{L^2}$ and $\|\sigma\|_{L^2}$ remain uniformly bounded in time. In particular, because $c_i$ are nonnegative, it follows from the boundedness of $\|\sigma\|_{L^2}$ that $\|c_i\|_{L^2}$ are uniformly bounded in time for each $i$. This result replaces Proposition \ref{L2'}.
\end{proof}

\section{Global Regularity for Mixed Boundary Conditions}\la{mbc}
In this section, we again consider NPNS and NPS for two oppositely charged species. Suppose $\pa\Omega$ represents a cation selective membrane which allows for permeation of cations but blocks anions. As previously mentioned, ion selectivity is typically modelled by Dirichlet boundary conditions \cite{davidson,rubishtil}. So in this case, the boundary conditions for $c_i$, assuming $z_1>0>z_2$, are
\begin{align}
{c_1(x,t)}_{|\pa\Omega}=&\gamma_1>0\la{di}\\
(\pa_n c_2(x,t)+z_2c_2(x,t)\pa_n\Phi(x,t))_{|\pa\Omega}=&0\la{2bl}
\end{align}
where $\gamma_1$ is \textit{constant}. The boundary conditions for $\Phi$ and $u$ are unchanged (see (\ref{noslip}),(\ref{robin})).

As discussed in the introduction, such configurations are known, in general, to lead to electrokinetic instabilities whereby, for large enough voltage drops $\sup\xi -\inf\xi$, one starts seeing vortical flow patterns, and for even larger drops, even chaotic behavior, resembling fluid turbulence and reminiscent of thermal turbulence in Rayleigh-Bénard convection \cite{davidson}. Mathematically, the instability of the configuration considered in this section is manifested by the lack of a dissipative bound resembling that from Proposition \ref{diss}.

To prove global regularity in this mixed setting, we aim, as in the previous sections, to obtain a priori control of the growth of the strong norms, which characterize strong solutions, namely $c_i\in L^\infty_t H^1_x\cap L^2_tH^2_x,\, u\in L^\infty_t V\cap L^2_tH^2_x.$ Let us remark here that \textit{strong} regularity in fact implies $C^\infty((0,T];C^\infty(\bar\Omega))$ regularity so long as the boundary and boundary conditions are smooth, as is the case here. This equivalence follows from a bootstrapping scheme using standard parabolic theory (see, for example, \cite{evans}). This observation justifies the use of Sard's theorem \cite{jl} in the proof of the following proposition.

\begin{prop}
Let $(c_i,\Phi,u)$ be a strong solution of NPNS (\ref{np}),(\ref{pois}),(\ref{nse}) or NPS (\ref{np}),(\ref{pois}),(\ref{stokes}) for two oppositely charged species ($m=2, z_1>0>z_2$) on the time interval $[0,T]$ satisfying boundary conditions (\ref{di}),(\ref{2bl}),(\ref{noslip}),(\ref{robin}) and initial conditions $0\le c_i(0)\in H^1,\,u(0)\in V$. Then there exist ${\tilde M_2},\tilde m_2>0$ depending on the parameters of the system and the initial conditions such that for each $i=1,2$
\be
\sup_{t\in[0,T]}\|c_i(t)\|_{L^2}<\tilde M_2 e^{\tilde m_2T^3}=M_2(T).\la{M2t}
\ee\la{gL2}
\end{prop}
\begin{proof}
\textbf{Step 1. $L^\infty_t L^1_x$ bounds on $c_i$.} Integrating (\ref{np}) for $i=2$ over $\Omega$, it follows from (\ref{2bl}) that $\|c_2(t)\|_{L^1}=\|c_2(0)\|_{L^1}$ for all $t\ge 0$. To obtain control of $\|c_1(t)\|_{L^1}$, we proceed as follows. We fix a $C^1$, nondecreasing function $\chi:\mathbb{R}\to[0,\infty)$ such that $\chi(s)=0$ for $s\in (-\infty,0]$, $\chi(s)=s^2$ for $s\in (0,\fr{1}{2})$, $\chi(s)=1$ for $s\in[1,\infty)$, and $\chi'\le 2$ everywhere. We note that $(\chi(s))^\alpha$ converges pointwise to the indicator function ${\bf{1}}_{s>0}(s)$ as $\alpha\to 0^+.$

Now we fix a time $t>0$, and we note that since the mapping $x\in\bar\Omega\mapsto  {\bar c}_1(t,x):=c_1(t,x)-\gamma_1$ is smooth, it follows from Sard's theorem that there exists a sequence $\beta_n\to 0$ such that each $\beta_n$ is a regular value of ${\bar c}_1(t):\bar\Omega\mapsto\mathbb{R}$. In particular, it follows that the sets $\{{\bar c}_1(t)>\beta_n\}$ are (possibly empty) smooth submanifolds of $\Omega$ such that if $\{{\bar c}_1(t)>\beta_n\}$ is nonempty, then it has smooth boundary $\{{\bar c}_1(t)=\beta_n\}\subset\Omega$. 

Now we multiply (\ref{np}) for $i=1$ by the test function
\be
\psi=\psi_{\alpha,n}=(\chi\circ( {\bar c}_1(t)-\beta_n))^\alpha,\quad 0<\alpha<1
\ee
and integrate; then, defining the following primitive,
\be
Q(y)=Q_{\alpha}(y)=\int_0^y (\chi(s))^\alpha\,ds\la{Qa}
\ee 
we have on the left hand side, using $\div u=0$,
\be
\int_\Omega (\pa_t c_1)\psi\,dx+\int_\Omega u\cdot\na ( Q\circ({\bar c}_1(t)-\beta_n))\,dx=\int_\Omega (\pa_t c_1)\psi\,dx.
\ee
On the right hand side, we obtain, using the fact that $\psi=0$ on $\pa\{{\bar c}_1(t)>\beta_n\}=\{ {\bar c}_1(t)=\beta_n\}$
\be
\bal
D_1\int_{ {\bar c}_1(t)>\beta_n}\div(\na c_1+c_1\na\Phi)\psi\,dx=&-\alpha D_1\int_{ {\bar c}_1(t)>\beta_n}|\na c_1|^2\fr{\chi'\circ({\bar c}_1(t)-\beta_n)}{(\chi\circ({\bar c}_1(t)-\beta_n))^{1-\alpha}}\,dx\\
&+D_1\int_{{\bar c}_1(t)>\beta_n}\div (({\bar c}_1(t)-\beta_n)\na\Phi)\psi\,dx\\
&-\fr{D_1}{\epsilon}(\gamma_1+\beta_n)\int_{{\bar c}_1(t)>\beta_n}\rho\psi\,dx\\
\le&-\alpha D_1\int_{{\bar c}_1(t)>\beta_n} ({\bar c}_1(t)-\beta_n)\na\Phi\cdot\na c_1\fr{\chi'\circ({\bar c}_1(t)-\beta_n)}{(\chi\circ({\bar c}_1(t)-\beta_n))^{1-\alpha}}\,dx\\
&-\fr{D_1}{\epsilon}(\gamma_1+\beta_n)\int_{{\bar c}_1(t)>\beta_n}\rho\psi\,dx.
\eal
\ee
Thus far, we have
\be
\bal
\int_\Omega (\pa_t c_1)\psi\,dx\le& -\alpha D_1\int_{{\bar c}_1(t)>\beta_n} ({\bar c}_1(t)-\beta_n)\na\Phi\cdot\na c_1\fr{\chi'\circ({\bar c}_1(t)-\beta_n)}{(\chi\circ({\bar c}_1(t)-\beta_n))^{1-\alpha}}\,dx\\
&-\fr{D_1}{\epsilon}(\gamma_1+\beta_n)\int_{{\bar c}_1(t)>\beta_n}\rho\psi\,dx\la{g1}
\eal
\ee
Recalling the dependence of $\psi$ on $n$ and taking the limit of (\ref{g1}) as $n\to \infty$, we obtain the following estimate, which holds for each $t>0$
\be\bal
\int_\Omega (\pa_t c_1)(\chi\circ {\bar c}_1(t))^\alpha\,dx\le&-\alpha D_1\int_{{\bar c}_1(t)>0}{\bar c}_1(t)\na\Phi\cdot\na c_1\fr{\chi'\circ {\bar c}_1(t)}{(\chi\circ {\bar c}_1(t))^{1-\alpha}}\,dx\\
&-\fr{D_1}{\epsilon}\gamma_1\int_{{\bar c}_1(t)>0}\rho(\chi\circ {\bar c}_1(t))^\alpha\,dx.\la{g2}
\eal\ee
We note that the sequence $\beta_n$ depended on the fixed time $t$; however upon taking the limit $\beta_n$, the resulting bound (\ref{g2}) holds for all positive time. The limit on the right hand side is justified by dominated convergence along with the fact that, by our choice of $\chi$, we have
\be
\bal
{\bf{1}}_{\{{\bar c}_1(t)-\beta_n>0\}}|{\bar c}_1(t)-\beta_n|\fr{\chi'\circ ({\bar c}_1(t)-\beta_n)}{(\chi\circ ({\bar c}_1(t)-\beta_n))^{1-\alpha}}\le&{\bf{1}}_{\{0<{\bar c}_1(t)-\beta_n<\fr{1}{2}\}}|{\bar c}_1(t)-\beta_n|\fr{2|{\bar c}_1(t)-\beta_n|}{|{\bar c}_1(t)-\beta_n|^{2-2\alpha}}\\
&+{\bf{1}}_{\{\fr{1}{2}\le {\bar c}_1(t)-\beta_n\}}\fr{2\|{\bar c}_1(t)-\beta_n\|_{L^\infty}}{(1/2)^{2-2\alpha}}\\
\le&8\max\{\|{\bar c}_1(t)-\beta_n\|_{L^\infty}^{2\alpha},\|{\bar c}_1(t)-\beta_n\|_{L^\infty}\}.\la{g3}
\eal
\ee
The above bound ensures that the integrand of the first integral on the right hand side of (\ref{g1}) is bounded uniformly in $n$.

We now bound the right hand side of (\ref{g2}). The first integral is bounded as follows
\be\bal
\left|\alpha D_1\int_{{\bar c}_1(t)>0}{\bar c}_1(t)\na\Phi\cdot\na c_1\fr{\chi'\circ {\bar c}_1(t)}{(\chi\circ {\bar c}_1(t))^{1-\alpha}}\,dx\right|\le& 2\alpha D_1\left|\int_{0<{\bar c}_1(t)<\fr{1}{2}}{\bar c}_1(t)|\na\Phi||\na c_1|\fr{{\bar c}_1(t)}{({\bar c}_1(t))^{2-2\alpha}}\,dx\right|\\
&+\alpha D_1\left|\int_{\fr{1}{2}\le {\bar c}_1(t)}{\bar c}_1(t)|\na\Phi||\na c_1|\fr{2}{(1/2)^2}\,dx \right|\\
\le& 2\alpha D_1\|\na\Phi\|_{L^\infty}\|\na c_1\|_{L^\infty}(1/2)^{2\alpha}|\Omega|\\
&+8\alpha D_1\|\na\Phi\|_{L^\infty}\|\na c_1\|_{L^\infty}\|{\bar c}_1(t)\|_{L^1}\\
\le&\alpha D_1\|\na\Phi\|_{L^\infty}\|\na c_1\|_{L^\infty}(2|\Omega|+8\|{\bar c}_1(t)\|_{L^1}).\la{bbb}
\eal\ee
Using the nonnegativity of $c_1$, the second integral is bounded by
\be
-\fr{D_1}{\epsilon}\gamma_1\int_{{\bar c}_1(t)>0}\rho(\chi\circ {\bar c}_1(t))^\alpha\,dx\le \fr{D_1}{\epsilon}\gamma_1\|c_2(t)\|_{L^1}=\fr{D_1}{\epsilon}\gamma_1\|c_2(0)\|_{L^1}\la{bbb'}
\ee

Now returning to (\ref{g2}) and writing it as
\be
\fr{d}{dt}\int_\Omega Q\circ {\bar c}_1(t)\,dx\le-\alpha D_1\int_{{\bar c}_1(t)>0}{\bar c}_1(t)\na\Phi\cdot\na c_1\fr{\chi'\circ {\bar c}_1(t)}{(\chi\circ {\bar c}_1(t))^{1-\alpha}}\,dx-\fr{D_1}{\epsilon}\gamma_1\int_{{\bar c}_1(t)>0}\rho(\chi\circ {\bar c}_1(t))^\alpha\,dx
\ee
we find, using (\ref{bbb}), (\ref{bbb'}) and integrating in time,
\be
\bal
\int_\Omega Q\circ {\bar c}_1(t)\,dx\le& \int_\Omega Q\circ {\bar c}_1(0)\,dx+\alpha D_1\int_0^t \|\na\Phi(s)\|_{L^\infty}\|\na c_1(s)\|_{L^\infty}(2|\Omega|+8\|{\bar c}_1(s)\|_{L^1})\,ds\\
&+\fr{D_1}{\epsilon}\gamma_1\|c_2(0)\|_{L^1}t.\la{g4}
\eal
\ee
Now we recall the dependence of $Q$ on $\alpha$ and observe that $Q=Q_\alpha(y)\to y_+:=\max\{y,0\}$ as $\alpha\to 0$. Using these facts, we take the limit of (\ref{g4}) as $\alpha\to 0$ to obtain
\be\bal
\int_\Omega {\bar c}_1(t)_+\,dx\le&\int_\Omega {\bar c}_1(0)_+\,dx+\fr{D_1}{\epsilon}\gamma_1\|c_2(0)\|_{L^1}t=\tilde B(t).\la{g5}
\eal\ee
Thus, since,
\be
\|{c}_1(t)\|_{L^1}\le \int_\Omega {\bar c}_1(t)_++\gamma_1\,dx\le \tilde B(t)+\gamma_1|\Omega|=B(t)\la{B(t)}
\ee
we have shown that $\|c_1(t)\|_{L^1}$ grows at most linearly in time.

\textbf{Step 2. Time integrability of $\|\rho\|_{L^2}^2$.} Unlike in Section \ref{UL2}, it is necessary to treat $c_1,c_2$ separately as they satisfy different boundary conditions. We first consider $c_1$. Multiplying (\ref{np}) for $i=1$ by $\fr{1}{D_1}(\log c_1-\log\gamma_1)$ and integrating by parts, we obtain 
\be
\bal
\fr{1}{D_1}\fr{d}{dt}\left(\int_\Omega c_1\log c_1-c_1-c_1\log\gamma_1\,dx\right)=&-\int_\Omega \fr{|\na c_1|^2}{c_1}+z_1\na\Phi\cdot\na c_1\,dx\\
=&-4\int_\Omega |\na \sqrt{c_1}|^2\,dx-z_1\int_\Omega \na\Phi\cdot\na(c_1-\gamma_1)\,dx\\
=&-4\int_\Omega |\na \sqrt{c_1}|^2\,dx+\fr{z_1}{\epsilon}\gamma_1\int_{\Omega}\rho\,dx-\fr{z_1}{\epsilon}\int_\Omega c_1\rho\,dx\la{logc1}
\eal
\ee
We observe that no boundary terms occur because $c_1=\gamma_1$ on the boundary, and the advective term involving $u$ also vanishes because $\div u=0$ and because $\log\gamma_1$ is a constant. 

Similarly, multiplying (\ref{np}) for $i=2$ by $\fr{1}{D_2}\log c_2$ and integrating by parts, we obtain
\be
\bal
\fr{1}{D_2}\fr{d}{dt}\left(\int_\Omega c_2\log c_2-c_2\,dx\right)=&-\int_\Omega \fr{|\na c_2|^2}{c_2}+z_2\na\Phi\cdot\na c_2\,dx\\
=&-4\int_\Omega |\na \sqrt{c_2}|^2\,dx-z_2\int_{\pa\Omega}c_2\pa_n\Phi\,dS-\fr{z_2}{\epsilon}\int_\Omega c_2\rho\,dx.\la{logc2}
\eal
\ee
We bound the boundary term using the Robin boundary conditions (\ref{robin}) and Lemma \ref{trace} in the Appendix (specifically (\ref{T1}) with $p=2$), 
\be
\bal
\left|z_2\int_{\pa\Omega} c_2\pa_n\Phi\,dS\right|\le & \tau|z_2|\left|\int_{\pa\Omega}c_2\Phi\,dS\right|+\|\xi\|_{L^\infty{(\pa\Omega})}|z_2|\int_{\pa\Omega}c_2\,dS\\
\le&C(\|\Phi\|_{L^\infty(\Omega)}+\|\xi\|_{L^\infty(\pa\Omega)})\|\sqrt{c_2}\|_{L^2(\pa\Omega)}^2\\
\le&C(\|\rho\|_{L^\fr{7}{4}}+1)(\|\sqrt{c_2}\|_{L^2}^\fr{1}{2}\|\na \sqrt{c_2}\|_{L^2}^\fr{1}{2}+\|\sqrt{c_2}\|_{L^2})^2\\
\le&C(\|\rho\|_{L^1}^\fr{1}{7}\|\rho\|_{L^2}^\fr{6}{7}+1)(\|c_2\|_{L^1}^\fr{1}{2}\|\na\sqrt{c_2}\|_{L^2}+\|c_2\|_{L^1})\\
\le&2\|\na \sqrt{c_2}\|_{L^2}^2+\fr{1}{2\epsilon}\|\rho\|_{L^2}^2+C(\|\rho\|_{L^1}^2\|c_2\|_{L^1}^7+\|c_2\|_{L^1}+\|\rho\|_{L^1}^\fr{1}{4}\|c_2\|_{L^1}^\fr{7}{4}).\la{c2bdb}
\eal
\ee
The third line follows from the embedding $W^{2,\alpha}(\Omega)\hookrightarrow L^\infty(\Omega)$ for $\alpha>\fr{3}{2}$ (here, for concreteness we choose $\alpha=\fr{7}{4}$) and the fact that $\Phi$ is related to $\rho$ via the Poisson equation (\ref{pois}) and the boundary conditions (\ref{robin}). The fourth line follows from interpolating $L^\fr{7}{4}$ between $L^1$ and $L^2$. The last line follows from Young's inequalities.  

Defining 
\be
\mathcal{E}=\fr{1}{D_1}\int_\Omega c_1\log c_1-c_1-c_1\log\gamma_1\,dx+\fr{1}{D_2}\int_\Omega c_2\log c_2-c_2\,dx
\ee
we obtain by adding (\ref{logc1}) to (\ref{logc2}), using (\ref{c2bdb}), and recalling $\rho=z_1c_1+z_2c_2$,
\be
\bal
\fr{d}{dt}\mathcal{E}+2\int_\Omega |\na\sqrt{c_1}|^2+|\na\sqrt{c_2}|^2\,dx+\fr{1}{2\epsilon}\int_\Omega \rho^2\,dx\le G(t)\la{lgt}
\eal
\ee
where 
\be
G(t)=C(\|\rho\|_{L^1}^2\|c_2\|_{L^1}^7+\|c_1\|_{L^1}+\|c_2\|_{L^1}+\|\rho\|_{L^1}^\fr{1}{4}\|c_2\|_{L^1}^\fr{7}{4})
\ee
Because $\|c_2(t)\|_{L^1}$ is constant and $\|c_1(t)\|_{L^1}$ grows at most linearly in time (\ref{B(t)}, we have that $G(t)$ grows at most quadratically in time and in particular is locally integrable. Thus integrating (\ref{lgt}), we obtain
\be
2\epsilon\mathcal{E}(T)+\int_0^T\|\rho(s)\|_{L^2}^2\,ds\le 2\epsilon\left(\mathcal{E}(0)+\int_0^TG(t)\,dt\right)=r(T).
\ee
Lastly, we observe that because $x\log x$ is superlinear, for any $c>0$ there exists $C>0$ depending only on $c$ such that $-C\le x\log x-cx$ for all $x>0$. It follows that $\mathcal{E}(T)$ is bounded below $\mathcal{E}(T)>-C_0$, with $C_0$ independent of $T$. Therefore
\be
\int_0^T\|\rho(s)\|_{L^2}^2\,ds\le r(T)+2\epsilon C_0=R(T).\la{RT}
\ee
We note that $R(T)$ increases at most like $T^3$.

\textbf{Step 3. $L^\infty_tL^2_x$ bounds for $c_i$.} Equipped with this time dependent bound, we proceed to control $c_i$ in $L^2$. Multiplying (\ref{np}) for $i=2$ by $\fr{|z_2|}{D_2}c_2$ and integrating by parts, we obtain exactly as in (\ref{ener}),
\be
\bal
\fr{|z_2|}{2D_2}\fr{d}{dt}\int_\Omega c_2^2\,dx=&-|z_2|\int_\Omega|\na c_2|^2\,dx-\fr{z_2|z_2|}{2\epsilon}\int_\Omega c_2^2\rho\,dx\\
&+z_2|z_2|\fr{\tau}{2}\int_{\pa\Omega}c_2^2\Phi\,dS- \fr{z_2|z_2|}{2}\int_{\pa\Omega}c_2^2\xi\,dS\\
=&-|z_2|\int_\Omega|\na c_2|^2\,dx-\fr{z_2|z_2|}{2\epsilon}\int_\Omega c_2^2\rho\,dx+I_1+I_2.\la{AA}
\eal
\ee
We bound the boundary integrals using the embedding $H^2\hookrightarrow L^\infty$ and Lemma \ref{trace} in the Appendix,
\be
\bal
|I_1|\le& C\|\Phi\|_{L^\infty(\Omega)}\|c_2\|_{L^2(\pa\Omega)}^2\\
\le&C(\|\rho\|_{L^2}+1)(\|c_2\|_{L^2}^\fr{1}{2}\|\na c_2\|_{L^2}^\fr{1}{2}+\|c_2\|_{L^2})^2\\
\le&\fr{|z_2|}{4}\|\na c_2\|_{L^2}^2+C(\|\rho\|_{L^2}^2+1)\|c_2\|_{L^2}^2.\la{BB}
\eal
\ee 
The last line follows from Young's inequalities. Similarly,
\be
\bal
|I_2|\le& C\|\xi\|_{L^\infty(\pa\Omega)}\|c_2\|_{L^2(\pa\Omega)}^2\\
\le&C\|\xi\|_{L^\infty(\pa\Omega)}(\|c_2\|_{L^2}^\fr{1}{2}\|\na c_2\|_{L^2}^\fr{1}{2}+\|c_2\|_{L^2})^2\\
\le&\fr{|z_2|}{4}\|\na c_2\|_{L^2}^2+C\|c_2\|_{L^2}^2.\la{CC}
\eal
\ee
From (\ref{AA}), using (\ref{BB}) and (\ref{CC}), we obtain
\be
\fr{|z_2|}{2D_2}\fr{d}{dt}\|c_2\|_{L^2}^2+\fr{|z_2|}{2}\|\na c_2\|_{L^2}^2\le C(\|\rho\|_{L^2}^2+1)\|c_2\|_{L^2}^2-\fr{z_2|z_2|}{2\epsilon}\int_\Omega c_2^2\rho\,dx.\la{DD}
\ee
Next we obtain the corresponding estimate for $c_1$. We observe that $q_1=c_1-\gamma_1$ satisfies
\be
\bal
\pa_t q_1+u\cdot\na q_1=&D_1\div(\na q_1+z_1q_1\na\Phi)-\fr{D_1z_1\gamma_1}{\epsilon}\rho.\la{q1}
\eal
\ee
Using the fact that $q_1$ vanishes on the boundary, we multiply (\ref{q1}) by $\fr{|z_1|}{D_1}q_1$ and integrate by parts
\be
\bal
\fr{|z_1|}{2D_1}\fr{d}{dt}\|q_1\|_{L^2}^2+|z_1|\|\na q_1\|_{L^2}^2=&-\fr{z_1|z_1|}{2}\int_\Omega \na q_1^2\cdot\na\Phi\,dx-\fr{z_1|z_1|\gamma_1}{\epsilon}\int_\Omega\rho q_1\,dx\\
\le&-\fr{z_1|z_1|}{2\epsilon}\int_\Omega q_1^2\rho\,dx+C\|\rho\|_{L^2}\|q_1\|_{L^2}.\la{EE}
\eal
\ee
Now we add (\ref{DD}) to (\ref{EE}) to obtain
\be
\bal
&\fr{d}{dt}\left(\fr{|z_1|}{2D_1}\|q_1\|_{L^2}^2+\fr{|z_2|}{2D_2}\|c_2\|_{L^2}^2\right)+|z_1|\|\na q_1\|_{L^2}^2+\fr{|z_2|}{2}\|\na c_2\|_{L^2}^2\\
\le&-\fr{1}{2\epsilon}\int_\Omega (z_1|z_1|q_1^2+z_2|z_2|c_2^2)\rho\,dx+C(\|\rho\|_{L^2}^2+1)\|c_2\|_{L^2}^2+C\|\rho\|_{L^2}\|q_1\|_{L^2}\\
\le&-\fr{1}{2\epsilon}\int_\Omega (z_1|z_1|q_1^2+z_2|z_2|c_2^2)\rho\,dx+C(\|\rho\|_{L^2}^2+1)(\|q_1\|_{L^2}^2+\|c_2\|_{L^2}^2)+C.\la{W}
\eal
\ee
Next, we observe that because $z_1>0>z_2$ and denoting $\sigma=z_1c_1+|z_2|c_2\ge 0$,
\be
\bal
(z_1|z_1|q_1^2+z_2|z_2|c_2^2)\rho=(z_1^2q_1^2-z_2^2c_2^2)\rho=&(z_1q_1+|z_2|c_2)(z_1q_1+z_2c_2)\rho\\
=&(z_1c_1-z_1\gamma_1+|z_2|c_2)(z_1c_1-z_1\gamma_1+z_2c_2)\rho\\
=&(\sigma-z_1\gamma_1)(\rho-z_1\gamma_1)\rho\\
=&\rho^2\sigma-2z_1^2\gamma_1c_1\rho+z_1^2\gamma_1^2\rho\\
\ge&-2z_1^2\gamma_1c_1\rho+z_1^2\gamma_1^2\rho.\la{r3}
\eal
\ee
We note that in the last line, there is no cubic term. So using (\ref{r3}) together with Young's inequalities, we obtain from (\ref{W}),
\be
\bal
\fr{d}{dt}\mathcal{F}\le& C(\|\rho\|_{L^2}^2+1)+C(\|\rho\|_{L^2}^2+1)(\|q_1\|_{L^2}^2+\|c_2\|_{L^2}^2)\\
\le& C(\|\rho\|_{L^2}^2+1)+C(\|\rho\|_{L^2}^2+1)\mathcal{F}
\eal
\ee
for
\be
\mathcal{F}=\fr{|z_1|}{2D_1}\|q_1\|_{L^2}^2+\fr{|z_2|}{2D_2}\|c_2\|_{L^2}^2.
\ee
Thus recalling (\ref{RT}), a Grönwall estimate gives us
\be
\sup_{t\in[0,T]}\mathcal{F}(t)\le \left(\mathcal{F}(0)+C(R(T)+T)\right)\exp\left(C(R(T)+T)\right). 
\ee
This completes the proof of the proposition.
\end{proof}

As in the case of blocking boundary conditions, the $L^2$ bounds obtained in the previous proposition are sufficient, in the mixed boundary conditions setting, to obtain control of $c_i$ in the space $L^\infty_tH_x^1\cap L_t^2H_x^2$. Since the $L^2$ bound (\ref{M2t}) is time dependent, all higher regularity bounds are also time dependent. In particular, even for NPS, no time independent bounds of the type (\ref{M'}) are available.

\begin{thm}
For initial conditions $0\le c_i(0)\in H^1$, $u(0)\in V$ and for all $T>0$, NPS (\ref{np}),(\ref{pois}),(\ref{stokes}) for two oppositely charged species ($m=2,\,z_1>0>z_2$) has a unique strong solution $(c_i,\Phi,u)$ on the time interval $[0,T]$ satisfying the boundary conditions (\ref{di}),(\ref{2bl}),(\ref{noslip}),(\ref{robin}). Under the same hypotheses, NPNS (\ref{np}),(\ref{pois}),(\ref{nse}) for two oppositely charged species has a unique strong solution on $[0,T]$ satisfying the initial and boundary conditions provided
\be
U(T)=\int_0^T\|u(s)\|_V^4\,ds<\infty.
\ee
Moreover, the solution to NPS satisfies  
\begin{align}
\sup_{t\in[0,T]}\|c_i(t)\|_{H^1}^2+\int_0^T\|c_i(s)\|_{H^2}^2\,ds&\le M(T)\la{M(T)}\\
\sup_{t\in[0,T]}\|u(t)\|_V^2+\int_0^T\|u(s)\|_{H^2}^2\,ds&\le B'(T)\la{B'T}
\end{align}
for time dependent constants $M(T), B'(T)$ depending also on the parameters of the system and the initial conditions. The solution to NPNS satisfies 
\begin{align}
\sup_{t\in[0,T]}\|c_i(t)\|_{H^1}^2+\int_0^T\|c_i(s)\|_{H^2}^2\,ds&\le M_U(T)\la{MU(T)}\\
\sup_{t\in[0,T]}\|u(t)\|_V^2+\int_0^T\|u(s)\|_{H^2}^2\,ds&\le B'_U(T)\la{B'UT}
\end{align}
for time dependent constants $M_U(T),B'_U(T)$ depending also on the parameters of the system, the initial conditions, and $U(T)$.\la{gr!!}
\end{thm}

\begin{rem}
{We show in the proof that $M(T),B'(T),M_U(T),B'_U(T)$ satisfy the following asymptotic bounds
\begin{align}
    M(T)\lesssim& \exp\exp\exp\exp(CT^3)\\
    B'(T)\lesssim& \exp\exp(CT^3)\\
    M_U(T)\lesssim&\exp\left(\exp\exp\exp(CT^3)\exp(CU(T))\right)\\
    B'_U(T)\lesssim&\exp\exp(CT^3)\exp(CU(T)).
\end{align}
These bounds follow from using \eqref{M2t} and bootstrapping.}
\end{rem}

\begin{proof}
As in the proofs of Theorems \ref{gr!} and \ref{grm}, global existence of strong solutions follows from the local existence theorem and the a priori bounds (\ref{M(T)})-(\ref{B'UT}).

\textbf{Step 1. $L^\infty_tL^4_x$ bounds for $c_i$. $L^\infty_t W^{1,\infty}_x$ bounds for $\Phi$. $L^\infty_tV\cap L^2_tH^2_x$ bounds for $u$.} From (\ref{M2t}) and Sobolev estimates, we obtain as in (\ref{P6}),
\be
\|\na\Phi\|_{L^6}\le P_6(T) \la{P6T}
\ee
for some time dependent constant $P_6(T)$, {which asymptotically satisfies $P_6(T)\lesssim e^{CT^3}$ (c.f. \eqref{M2t}).} Then multiplying (\ref{np}) for $i=2$ by $c_2^3$, and integrating by parts, we find (cf. (\ref{kk})),
\be
\bal
\fr{1}{4}\fr{d}{dt}\|c_2\|_{L^4}^4+\fr{3}{4}D_2\|\na c_2^2\|_{L^2}^2\le CP_6(T)\|c_2^2\|_{L^3}\|\na c_2^2\|_{L^2}
\eal
\ee
which, after interpolating $L^3$ between $L^2$ and $H^1$ and using a Young's inequality, gives
\be
\fr{1}{4}\fr{d}{dt}\|c_2\|_{L^4}^4+\fr{1}{4}D_2\|\na c_2^2\|_{L^2}^2\le CP_6(T)^4\|c_2\|_{L^4}^4.
\ee
Therefore, for a time dependent constant $M'_4(T)$ we have
\be
\|c_2\|_{L^4}\le M'_4(T)\la{M'4T},
\ee
{where $M_4'(T)\lesssim \exp\exp(CT^3)$}. Then multiplying (\ref{q1}) by $q_1^3$ and integrating by parts, we obtain
\be
\fr{1}{4}\fr{d}{dt}\|q_1\|_{L^4}^4+\fr{3}{4}D_1\|\na q_1^2\|_{L^2}^2\le CP_6(T)\|q_1^2\|_{L^3}\|\na q_1^2\|_{L^2}+C(1+\|q_1\|_{L^4}^4+\|c_2\|_{L^4}^4)
\ee
which gives after an interpolation,
\be
\fr{1}{4}\fr{d}{dt}\|q_1\|_{L^4}^4+\fr{1}{4}D_1\|\na q_1^2\|_{L^2}^2\le C(P_6(T)^4+1)\|q_1\|_{L^4}^4+C(1+\|c_2\|_{L^4}^4).
\ee
This, together with (\ref{M'4T}), allows us to conclude that there exists a time dependent constant $M_4(T)$ such that for each $i$,
\be
\|c_i\|_{L^4}\le M_4(T)\la{M4T}
\ee
{with $M_4(T)\lesssim \exp\exp(CT^3)$}. Furthermore, since $W^{2,4}\hookrightarrow W^{1,\infty}$, we have for a time dependent constant $P_\infty(T)$,
\be
\|\Phi\|_{W^{1,\infty}}\le P_\infty(T)\la{PiT}
\ee
{with $P_\infty(T)\lesssim \exp\exp(CT^3)$}. At this point, we have established enough bounds to mimic the derivations of (\ref{unps}), (\ref{nnn}) to obtain the bounds (\ref{B'T}), (\ref{B'UT}) {for $u$ with $B'(T)\lesssim \exp\exp(CT^3)$ and $B'_U(T)\lesssim \exp\exp(CT^3)\exp (CU(T))$.}

\textbf{Step 2. $L^\infty_tH^1_x\cap L^2_tH^2_x$ bounds for $c_1$.} Now we bound derivatives of $c_1$. First we multiply (\ref{q1}) by $-\D q_1$ and integrate by parts to obtain
\be
\bal
\fr{1}{2}\fr{d}{dt}\|\na q_1\|_{L^2}^2+D_1\|\D q_1\|_{L^2}^2\le& C\|u\|_V\|\na q_1\|^\fr{1}{2}_{L^2}\|\D q_1\|_{L^2}^\fr{3}{2}+CP_\infty(T)\|\na q_1\|_{L^2}\|\D q_1\|_{L^2}\\
&+C(M_4(T)^2+1)\|\D q_1\|_{L^2}+CM_2(T)\|\D q_1\|_{L^2}
\eal
\ee
so that from Young's inequalities, we obtain
\be
\fr{1}{2}\fr{d}{dt}\|\na q_1\|_{L^2}^2+\fr{D_1}{2}\|\D q_1\|_{L^2}^2\le C(T)+C(P_\infty(T)^2+\|u\|_V^4)\|\na q_1\|_{L^2}^2\la{naq1}
\ee
for a time dependent constant {$C(T)\lesssim \exp\exp(CT^3)$}. Using a Grönwall estimate, from (\ref{naq1}), we obtain (\ref{MU(T)}) for $c_1$ in the case of NPNS {with $M_U(T)\lesssim \exp\exp\exp(CT^3)\exp (CU(T))$}. In the case of NPS, (\ref{naq1}) together with (\ref{B'T}) gives us (\ref{M(T)}) for $c_1$ {with $M(T)\lesssim \exp\exp\exp(CT^3)$}.

\textbf{Step 3. $L^\infty_tH^1_x\cap L^2_tH^2_x$ bounds for $c_2$.} To obtain the corresponding bounds for $c_2$, we start with (\ref{dci}), which holds verbatim for $i=2$:
\be
\bal
\fr{1}{2}\fr{d}{dt}\|\na\tilde c_2\|_{L^2}^2+D_2\|\D\tilde c_2\|_{L^2}^2=&\int_\Omega (u\cdot\na\tilde c_2)\D\tilde c_2\,dx+D_2z_2\int_\Omega(\na\tilde c_2\cdot\na\Phi)\D\tilde c_2\,dx\\
&-z_2\int_\Omega((\pa_t+u\cdot\na)\Phi)\tilde c_2\D\tilde c_2\,dx\\
=& I_1+I_2+I_3.
\eal
\ee
The term $I_1$ is bounded identically (see (\ref{id})):
\be
|I_1|\le\fr{D_2}{4}\|\D \tilde c_2\|_{L^2}^2+C\|u\|_V^4\|\na\tilde c_2\|_{L^2}^2.
\ee
The term $I_2$ is also bounded identically (see (\ref{idd})), noting that this time we have a time dependent coefficient $C_g(T)$, depending on $P_\infty(T)$ {so that $C_g(T)\lesssim\exp\exp(CT^3)$}:
\be
|I_2|\le\fr{D_2}{4}\|\D\tilde c_2\|_{L^2}^2+C_g(T)\|\na\tilde c_2\|_{L^2}^2.
\ee
As for the term $I_3=I_3^1+I_3^2$ (see (\ref{split})), we note that
\be
\|\tilde{c}_2\|_{L^3}\le CM_4(T)e^{|z_2|P_\infty(T)}=\beta_3(T){\lesssim\exp\exp\exp(CT^3)}
\ee
so that
\be
\bal
|I_3^1|\le& C\|u\|_V\|\na\Phi\|_{L^\infty}\|\tilde{c}_2\|_{L^3}\|\D \tilde{c}_2\|_{L^2}\\
\le&\fr{D_2}{8}\|\D\tilde{c}_2\|_{L^2}^2+C'_g(T)\|u\|_V^2
\eal
\ee
where {$C'_g(T)\lesssim \exp\exp\exp(CT^3)$.} Lastly, in order to bound $I_3^2$, we first estimate $\pa_t\na\Phi$ as in (\ref{ptrho})-(\ref{ptg}), but this time we work around the fact that ${\pa_n \tilde c_1}_{|\pa\Omega}=0$ does not hold:
\be
\bal
\tau\|\pa_t\Phi\|&_{L^2(\pa\Omega)}^2+\|\pa_t\na\Phi\|_{L^2(\Omega)}^2\\\le& C\sum_{j=1}^2\left|\int_\Omega\div(e^{-z_j\Phi}\na\tcj)\pa_t\Phi\,dx\right|+C\left|\int_\Omega\div(u\rho)\pa_t\Phi\,dx\right|\\
\le&C\int_\Omega e^{|z_2|P_\infty(T)}|\na \tilde c_2||\pa_t\na\Phi|\,dx+C\left|\int_\Omega \div (e^{-z_1\Phi}\na\tilde c_1)\pa_t\Phi\,dx\right|\\
&+C\int_\Omega |\rho||u||\pa_t\na\Phi|\,dx\\
\le&C\int_\Omega e^{|z_2|P_\infty(T)}|\na \tilde c_2||\pa_t\na\Phi|\,dx+CP_\infty(T) e^{|z_1|P_\infty(T)}\int_\Omega  |\na\tilde c_1||\pa_t\Phi|\,dx\\
&+Ce^{|z_1|P_\infty(T)}\int_\Omega |\D\tilde c_1||\pa_t\Phi|\,dx+C\int_\Omega |\rho||u||\pa_t\na\Phi|\,dx\\
\le&C_{t}(T)(\sum_{j=1}^2\|\na\tcj\|_{L^2}+\|\D\tilde c_1\|_{L^2}+\|u\|_V)(\|\pa_t\na\Phi\|_{L^2(\Omega)}^2+\|\pa_t\Phi\|_{L^2(\pa\Omega)}^2)^\fr{1}{2}
\eal
\ee
where {$C_t(T)\lesssim \exp\exp\exp(CT^3)$}, and we used $\|\pa_t\Phi\|_{L^2}\le C(\|\pa_t\na\Phi\|_{L^2(\Omega)}^2+\|\pa_t\Phi\|_{L^2(\pa\Omega)}^2)^\fr{1}{2}$. Then recalling that we already have the bounds (\ref{M(T)}), (\ref{MU(T)}) for $c_1$, we obtain as in (\ref{pt6}), using the embedding $H^1\hookrightarrow L^6$,
\be
\|\pa_t\Phi\|_{L^6}\le C_{t}(T,U(T))(\|\na\tilde{c}_2\|_{L^2}+\|\D\tilde c_1\|_{L^2}+\|u\|_V+1)\la{pt6t}
\ee
for a time dependent constant {$C_t(T,U(T))\lesssim \exp\exp\exp(CT^3)\exp(CU(T))$ (though in the case of NPS, the dependence on $U(T)$ is redundant and may be dropped)}. Thus we bound $I_3^2$ as in (\ref{I32}),
\be
\bal
|I_3^2|\le\fr{D_2}{8}\|\D\tilde c_2\|_{L^2}^2+C_3(T,U(T))(\|\na\tilde{c}_2\|_{L^2}^2+\|\D\tilde c_1\|_{L^2}^2+\|u\|_V^2+1)\la{I32t}
\eal
\ee
{with $C_3(T,U(T))\lesssim \exp\exp\exp(CT^3)\exp(CU(T))$}. Collecting our estimates for $I_1,I_2,I_3$, we obtain as in (\ref{cru}),
\be
\bal
&\fr{1}{2}\fr{d}{dt}\|\na \tilde{c}_2\|_{L^2}^2+\fr{D_2}{4}\|\D\tilde{c}_2\|_{L^2}^2\\
\le& (C_F(T,U(T))+C\|u\|_V^4)\|\na \tilde{c}_2\|_{L^2}^2+C'_F(T,U(T))(\|\D\tilde c_1\|_{L^2}^2+\|u\|_V^2+1)\la{crut}
\eal
\ee
{with $C_F(T,U(T)),C_F'(T,U(T))\lesssim \exp\exp\exp(CT^3)\exp(CU(T))$}. It is straightforward to verify using the chain rule that
\be
\|\D\tilde c_1\|_{L^2}\le C_P(T)(\|\D c_1\|_{L^2}+\|\na c_1\|_{L^2}+\|c_1\|_{L^2}+\|c_1\rho\|_{L^2})
\ee
for a time dependent constant $C_P(T)$ depending on $P_\infty(T)$ {so that $C_P(T)\lesssim \exp\exp\exp(CT^3)$}. So, by the previously established bounds (\ref{M(T)}), (\ref{MU(T)}) for $c_1$ (cf. (\ref{naq1})), together with (\ref{M4T}), we have
\be
\int_0^T\|\D\tilde c_1\|_{L^2}^2\,dx={S(T,U(T))\lesssim \exp\exp\exp(CT^3)\exp(CU(T))} .\la{ST}
\ee
Then from (\ref{crut}), a Grönwall estimate together with (\ref{ST}) gives us (\ref{MU(T)}) for $i=2$ in the case of NPNS, after we convert back to the variable $c_2$, with $M_U(T)\lesssim\exp(\exp\exp\exp(CT^3)\exp(CU(T)))$. In the case of NPS, (\ref{B'T}) together with (\ref{crut}) and (\ref{ST}) gives us (\ref{M(T)}) for $i=2$ {with $M_U(T)\lesssim \exp\exp\exp\exp(CT^3)$}. This completes the proof.
\end{proof}

\appendix
\section{Trace Inequalities}
\begin{lemma}\la{trace}
Let $\Omega\subset\mathbb{R}^3$ be an open, bounded domain with Lipschitz boundary. Then for $p\in[2,4]$, we have the embedding $H^1(\Omega)\hookrightarrow L^p(\pa\Omega).$ Moreover, for $p\in [2,4]$, there exist constants $C_p, C'_p$ depending only on $\Omega$ and $p$ such that
\be
\|f\|_{L^p(\pa\Omega)}\le C_p\|\na f\|_{L^2(\Omega)}^\fr{1}{p}\|f\|_{L^{2(p-1)}(\Omega)}^\fr{p-1}{p}+C'_p\|f\|_{L^p(\Omega)}.\la{T1}
\ee
In particular, for $p=4$, there exists a constant $c_4$ depending only on $\Omega$ such that 
\be
\|f\|_{L^4(\pa\Omega)}\le c_4\|f\|_{H^1(\Omega)}.\la{T2}
\ee
And, for $p\in[2,4)$ and any $\gamma>0$ there exists  $C_\gamma$ depending on $\Omega,p$ and $\gamma$ such that 
\be
\|f\|_{L^p(\pa\Omega)}^2\le \gamma\|\na f\|_{L^2(\Omega)}^2+ C_\gamma\|f\|_{L^1(\Omega)}^2.\la{T3}
\ee

\begin{proof}
The proof is based on the proof of Theorem 1.5.1.10 in \cite{grisvard}. Because $C^1(\bar\Omega)$ is dense in $H^1(\Omega)$, we assume without loss of generality that $f\in C^1(\bar\Omega).$ Next we use the fact (Lemma 1.5.1.9, \cite{grisvard}) that for bounded Lipschitz domains $\Omega\subset\mathbb{R}^3$, there exists $\mu\in (C^\infty(\bar\Omega))^3$ and a constant $\delta>0$ such that $\mu\cdot n\ge \delta$ almost everywhere on $\pa\Omega$, where $n$ is the outward pointing normal vector on $\pa\Omega$. Then, on one hand, we have
\be
\int_\Omega \na |f|^p\cdot\mu\,dx=p\int_\Omega |f|^{p-2}f\na f\cdot\mu\,dx.
\ee
On the other hand, integrating by parts, we have
\be
\int_\Omega \na |f|^p\cdot\mu\,dx=\int_{\pa\Omega} |f|^p\mu\cdot n\,dS-\int_\Omega |f|^p\div\mu\,dx.
\ee
Therefore,
\be
\bal
\delta\int_{\pa\Omega}|f|^p\,dS\le&\int_{\pa\Omega} |f|^p\mu\cdot n\,dS\\
=&p\int_\Omega |f|^{p-2}f\na f\cdot\mu\,dx+\int_\Omega |f|^p\div\mu\,dx\\
\le& p\|\mu\|_{L^\infty(\Omega)}\int_\Omega |f|^{p-1}|\na f|\,dx+\|\div\mu\|_{L^\infty(\Omega)}\int_\Omega|f|^p\,dx\\
\le& p\|\mu\|_{L^\infty(\Omega)}\|\na f\|_{L^2(\Omega)}\|f\|_{L^{2(p-1)}(\Omega)}^{p-1}+\|\div\mu\|_{L^\infty(\Omega)}\|f\|_{L^p(\Omega)}^p\la{tt}
\eal
\ee
where the last line follows from a Hölder inequality. Then, (\ref{T1}) follows from taking both sides of (\ref{tt}) to the $p^{-1}$ power. Taking $p=4$ in (\ref{T1}) we have
\be
\|f\|_{L^4(\pa\Omega)}\le C_4\|\na f\|_{L^2(\Omega)}^\fr{1}{4}\|f\|_{L^6(\Omega)}^\fr{3}{4}+C'_4\|f\|_{L^4(\Omega)},
\ee
so (\ref{T2}) follows from the embedding $H^1(\Omega)\hookrightarrow L^6(\Omega)$. Lastly, (\ref{T3}) follows from (\ref{T1}) by interpolating the spaces $L^{2(p-1)}(\Omega)$ and $L^p(\Omega)$ between $L^1(\Omega)$ and $H^1(\Omega)$, followed by Young's inequalities.
\end{proof}
\end{lemma}

\section{Positivity of Concentrations}\la{pc}
In this section, we verify that positivity of the ionic concentrations, $c_i$, is propagated in time by the Nernst-Planck equations.
\begin{prop}
Suppose $(c_1,c_2,u)$ is a strong solution to NPNS (\ref{np}), (\ref{pois}), (\ref{nse}) or NPS (\ref{np}), (\ref{pois}), (\ref{stokes}) with blocking boundary conditions (\ref{bl}), (\ref{noslip}), (\ref{robin}) or mixed boundary conditions (\ref{di}), (\ref{2bl}), (\ref{noslip}), (\ref{robin}) with strictly positive initial conditions $0<c\le c_i(0)\in H^1(\Omega)$. Then, $c_i(t)>0$ for all $t\ge 0$.
\end{prop}
\begin{proof}
We split the proof into two steps. First we show that $c_i(t)\ge 0$ for all $t\ge 0$. This part only uses $c_i(0)\ge 0$ and can be shown as in \cite{ci,cil}. We fix a convex function on the real line that is positive on the negative semiaxis and zero on the positive semiaxis. For example, we take the function
\be
F(y)=\begin{cases}
y^4,& y<0\\
0,&y\ge 0.
\end{cases}
\ee
We observe that $F$ satisfies, for all $y\in\mathbb{R}$,
\be
F''(y)y^2\le 12 F(y).
\ee
Now we multiply (\ref{np}) by $F'(c_i)$ and integrate by parts, noting that due to the choice of $F$, no boundary terms occur for blocking boundary conditions nor for mixed boundary conditions:
\be
\bal
\fr{d}{dt}\int_\Omega F(c_i)\,dx=&-D_i\int_\Omega|\na c_i|^2 F''(c_i)\,dx- D_iz_i\int_\Omega c_iF''(c_i)\na c_i\cdot\na\Phi\,dx\\
\le&-\fr{D_i}{2}\int_\Omega |\na c_i|^2F''(c_i)\,dx+\fr{D_iz_i^2}{2}\int_\Omega F''(c_i)c_i^2|\na\Phi|^2\,dx\\
\le&6D_i\|\na\Phi\|_{L^\infty}^2\int_\Omega F(c_i)\,dx.\la{FF}
\eal
\ee
We note that the advective term involving $u$ vanishes due to $\div u=0$. It follows from (\ref{FF}) that
\be
\int_\Omega F(c_i(t))\,dx\le \left(\int_\Omega F(c_i(0))\,dx\right)\exp\left(6D_i\int_0^t\|\na\Phi(s)\|_{L^\infty}^2\,ds\right).
\ee
For strong solutions, the time integral
\be
\int_0^t\|\na\Phi(s)\|_{L^\infty}^2\,ds
\ee
is finite, and thus since $F(c_i(0))=0$ on $\Omega$, it follows that
\be
\int_\Omega F(c_i(t))=0
\ee
which implies $c_i(t)\ge 0$. This proves the nonnegativity of $c_i$.

Improving this result to strict positivity requires an additional argument. We adapt the argument given in \cite{gaj}. We first fix a time interval $[0,T]$, and we assume that $c_i$ satisfy blocking boundary conditions. We then fix $\delta>0$ and multiply (\ref{np}) by $-\fr{1}{(c_i+\delta)^2}$ and integrate by parts. Then on the left hand side, we obtain
\be
\fr{d}{dt}\int_\Omega \fr{1}{c_i+\delta}\,dx.
\ee
On the right hand side, we have
\be
-2D_i\int_\Omega \fr{|\na c_i|^2}{(c_i+\delta)^3}\,dx+2D_iz_i\int_\Omega c_i\na\Phi\cdot\fr{\na c_i}{(c_i+\delta)^3}\,dx=I_1+I_2.
\ee
The integral $I_1$ is nonpositive and the second integral $I_2$ can be estimated as follows, using Young's inequality
\be
|I_2|\le D_i\int_\Omega \fr{|\na c_i|^2}{(c_i+\delta)^3}\,dx+C_T\int_\Omega \fr{1}{c_i+\delta}\,dx
\ee
where $C_T$ depends on parameters and on $\sup_{t\in [0,T]}\|\na\Phi(t)\|_{L^\infty}$ but not on $\delta$. Thus we have that $L_1:=\int_\Omega \fr{1}{c_i+\delta}\,dx$ satisfies $\fr{d}{dt}L_1\le C_TL_1$, and thus
\be
\sup_{t\in[0,T]}L_1(t)\le e^{C_T T}\int_\Omega \fr{1}{c_i(0)+\delta}\,dx\le e^{C_T T}\fr{|\Omega|}{c}\la{l1t}
\ee
where we note that the final upper bound does not depend on $\delta$.

Now the idea is to bootstrap to obtain bounds on $L_{2^k}:=\int_\Omega \fr{1}{(c_i+\delta)^{2^k}}$ for $k=1,2,3,...$ exactly as was done to control $\|c_i\|_{L^\infty}$ in Proposition \ref{L2}. 

To this end, we multiply (\ref{np}) by $\fr{-j+1}{(c_i+\delta)^{j}}$ ($j=3,4,...$) and integrate by parts. This yields
\be
\bal
\fr{d}{dt}\left\|\fr{1}{(c_i+\delta)^\fr{j-1}{2}}\right\|_{L^2}^2+4D_i\left\|\na\fr{1}{(c_i+\delta)^\fr{j-1}{2}}\right\|_{L^2}^2\le& D_i|z_i|(j-1)\left|\int_\Omega c_i\na\Phi\cdot\na (c_i+\delta)^{-j}\,dx\right|\\
\le& 2D_i|z_i|j\|\na\Phi\|_{L^\infty}\int_\Omega|\na (c_i+\delta)^\fr{-j+1}{2}|(c_i+\delta)^\fr{-j+1}{2}\,dx\\
\le& 2D_i\left\|\na\fr{1}{(c_i+\delta)^\fr{j-1}{2}}\right\|_{L^2}^2\\
&+2D_i|z_i|^2j^2\|\na\Phi\|_{L^\infty}^2\left\|\fr{1}{(c_i+\delta)^\fr{j-1}{2}}\right\|_{L^2}^2\la{agag}
\eal
\ee
Then, using the interpolation estimate
\be
\left\|\fr{1}{(c_i+\delta)^\fr{j-1}{2}}\right\|_{L^2}\le C\left(\left\|\na\fr{1}{(c_i+\delta)^\fr{j-1}{2}}\right\|_{L^2}^\fr{3}{5}\left\|\fr{1}{(c_i+\delta)^\fr{j-1}{2}}\right\|_{L^1}^\fr{2}{5}+\left\|\fr{1}{(c_i+\delta)^\fr{j-1}{2}}\right\|_{L^1}\right)\la{int1}
\ee
followed by a Young's inequality, we obtain from (\ref{agag})
\be
\fr{d}{dt}\left\|\fr{1}{(c_i+\delta)^\fr{j-1}{2}}\right\|_{L^2}^2+D_i\left\|\na\fr{1}{(c_i+\delta)^\fr{j-1}{2}}\right\|_{L^2}^2\le \bar{c}j^l\left\|\fr{1}{(c_i+\delta)^\fr{j-1}{2}}\right\|_{L^1}^2
\ee
for some $l>0$ large enough and $\bar c$ depending on parameters, the domain, $\sup_{t\in[0,T]}\|\na\Phi(t)\|_{L^\infty}$, but not on $j$. Then, again using the interpolation estimate  (\ref{int1}) followed by a Young's inequality, we obtain
\be
\fr{d}{dt}\left\|\fr{1}{(c_i+\delta)^\fr{j-1}{2}}\right\|_{L^2}^2\le -C\left\|\fr{1}{(c_i+\delta)^\fr{j-1}{2}}\right\|_{L^2}^2+ C_j\left\|\fr{1}{(c_i+\delta)^\fr{j-1}{2}}\right\|_{L^1}^2\la{jj}
\ee
where $C_j\le \tilde c j^l$ for some $\tilde c$ not depending on $j$. Now we set $j=2k+1$ where $k$ is a nonnegative integer, and we rewrite (\ref{jj}) as
\be
\fr{d}{dt}\left\|\fr{1}{c_i+\delta}\right\|_{L^{2k}}^{2k}\le -C \left\|\fr{1}{c_i+\delta}\right\|_{L^{2k}}^{2k} +C_{2k+1}\left\|\fr{1}{c_i+\delta}\right\|_{L^{k}}^{2k}.\la{2kk}
\ee
Then, applying Grönwall's inequality to (\ref{2kk}), we obtain as in (\ref{2k'})-(\ref{ss}), for some $\bar C>0$ independent of $k$,

\be
T_{2k}\le {\bar C}^\fr{1}{2k}k^\fr{l}{2k}T_k
\ee
where 
\be
T_k=\max\left\{\left\|\fr{1}{c_i(0)+\delta}\right\|_{L^\infty},\sup_{t\in [0,T]}\left\|\fr{1}{c_i(t)+\delta}\right\|_{L^k}\right\}.
\ee
Now setting $k=2^\kappa$ we have
\be
T_{2^{\kappa+1}}\le {\bar C}^\fr{1}{2^{\kappa+1}}2^\fr{\kappa l}{2^{\kappa+1}}T_{2^\kappa}
\ee
from which it follows that for all $J\in\mathbb{N}$
\be
T_{2^J}\le {\bar C}^a2^bT_1\la{t2j}
\ee
where
\be
a=\sum_{\kappa=0}^\infty \fr{1}{2^{\kappa+1}}<\infty,\quad b=\sum_{\kappa=0}^\infty\fr{\kappa l}{2^{\kappa+1}}<\infty.
\ee
So passing $J\to\infty$ in (\ref{t2j}) we find that
\be
\bal
\sup_{t\in[0,T]}\left\|\fr{1}{c_i(t)+\delta}\right\|_{L^\infty}\le& {\bar C}^a 2^b\max\left\{\left\|\fr{1}{c_i(0)+\delta}\right\|_{L^\infty},\sup_{t\in [0,T]}\left\|\fr{1}{c_i(t)+\delta}\right\|_{L^1}\right\}\\
\le&{\bar C}^a2^b\max\left\{\left\|\fr{1}{c_i(0)+\delta}\right\|_{L^\infty}, e^{C_T T}\fr{|\Omega|}{c}\right\}
\eal
\ee
where the second inequality follows from (\ref{l1t}. Finally, passing to the limit $\delta\to 0^+$
gives
\be
\sup_{t\in[0,T]}\left\|\fr{1}{c_i(t)}\right\|_{L^\infty}\le {\bar C}^a2^b\max\left\{\fr{1}{c}, e^{C_T T}\fr{|\Omega|}{c}\right\},
\ee
and thus on any finite time interval $c_i(t)$ is uniformly bounded away from $0$ on $\bar\Omega$. This completes the proof of strict positivity of $c_i, i=1,2$ in the case of blocking boundary conditions.

In the case of mixed boundary conditions (\ref{di}), (\ref{2bl}), the strict positivity of $c_2$ is obtained as in the blocking case. For $c_1$, it is possible to obtain strict positivity along same lines as the preceding proof for blocking boundary conditions by choosing appropriate test functions, and in fact, this method generalizes easily to more complex boundary conditions for $c_i$ (e.g. blocking boundary conditions on a nontrivial boundary portion, Dirichlet boundary conditions on the complement, see \cite{ci, np3d}) and also to cases of more than two ionic species. Here, since $c_1$ satisfies purely Dirichlet boundary conditions and because we are considering the case of two oppositely charged ionic species, we argue using a maximum principle argument (see \cite{cil} for a similar argument in the context of upper bounds). As per the remark at the start of Section \ref{mbc}, we assume that $c_i,u,\Phi$ are all smooth in both space and time so that the following considerations are justified.

We know that on the interval $[0,T]$, there exists $c_T>0$ such that
\be
c_T\le c_2(t).\ee
Now fix $0<c'<\min\{c, c_T,\gamma_1\}$. Initially, we have, by definition of $c$, that $c_1(0)> c'$ on $\bar\Omega$. Suppose $c_1$ attains the value $c'$ at some time in between $t=0$ and $t=T$. Suppose $t'$ is the \textit{first} time when $c'$ is attained by $c_1$. Then, using (\ref{pois}), we write (\ref{np}) for $i=1$ as
\be
\pa_t c_1+u\cdot\na c_1=D_1\D c_1+D_1\na c_1\cdot\na \Phi-\fr{D_1}{\epsilon}c_1(c_1-c_2).
\ee
Evaluating the above at time $t'$, at the point $x_0\in\Omega$ where the minimal value $c'$ is attained, we find that
\be
\pa_t c_1(t',x_0)\ge -\fr{D_1}{\epsilon}c'(c'-c_T)>0.
\ee
However, by our choice of $t'$ and $x_0$, we must have $\pa_t c_1(t',x_0)\le 0$, so we have a contradiction. Thus, \be\inf_{[0,T]\times\bar\Omega}c_1\ge \min\{c,c_T,\gamma_1\}>0.\ee Since $[0,T]$ is an arbitrary finite time interval, we have shown that $c_1$ is also uniformly bounded away from 0 on every finite time interval. This completes the proof of the strict positivity of $c_i, i=1,2$ in the case of mixed boundary conditions.
\end{proof}

\vspace{.5cm}

{\bf{Acknowledgment.}} The author thanks the anonymous referees for their constructive comments.

\end{document}